\documentclass[12pt]{amsart}
\usepackage{geometry}
\usepackage{amsmath}
\usepackage{amsthm}
\usepackage{bbm}
\usepackage{amsfonts,amssymb,bm}
\usepackage{fancyhdr}
\usepackage{mathrsfs}
\usepackage{appendix}
\usepackage{graphicx}
\usepackage[all]{xy}
\usepackage{color}
\usepackage{enumerate}

\usepackage{float} 
\usepackage{amsthm}

\newtheorem{thm}{Theorem}[section]
\newtheorem{defi}[thm]{Definition}
\newtheorem{lem}[thm]{Lemma}
\newtheorem{cor}[thm]{Corollary}

\newtheorem{rmk}[thm]{Remark}
\newtheorem{ques}[thm]{Question}

\usepackage{times}
\usepackage{hyperref}
\def\Aut{{\rm Aut}}
\def\Hom{{\rm Hom}}
\def\ord{{\rm ord}}
\def\min{{\rm min}}
\def\max{{\rm max}}
\def\inf{{\rm inf}}
\def\sup{{\rm sup}}
\def\lim{{\rm lim}}

\def\dif{{\rm d}}

\def\Supp{{\rm Supp}}
\def\gr{{\rm gr}}
\def\pr{{\rm pr}}
\def\wt{{\rm wt}}

\def\Proj{{\rm Proj}}

\def\triv{{\rm triv}}
\def\Bl{{\rm Bl}}
\def\log{{\rm log}}
\def\vol{{\rm vol}}
\def\Val{{\rm Val}}

\def\lct{{\rm lct}}
\def\DH{{\rm DH}}
\def\LE{{\rm LE}}

\def\Ex{{\rm Ex}}

\def\QM{{\rm QM}}
\def\LC{{\rm LC}}

\def\lc{{\rm lc}}
\def\ac{{\rm ac}}
\def\s{{\rm s}}
\def\Bs{{\rm Bs}}

\def\l0{{ l_0}}
\def\lN{{l_0\IN}}

\def\BP{\mathbf{P}}

\def\BO{\mathbf{O}}

\def\Bv{\mathbf{v}}
\def\BL{\mathbf{L}}
\def\BD{\mathbf{D}}

\def\BH{\mathbf{H}}
\def\BF{\mathbf{F}}

\def\B0{\mathbf{0}}

\def\Bin{\mathbf{in}}

\def\tbeta{\tilde{\beta}}
\def\tGamma{\widetilde{\Gamma}}

\def\bg{\bar{g}}
\def\bY{\overline{Y}}
\def\bxi{\bar{\xi}}

\def\opi{{\overline{\pi}}}
\def\oPhi{{\overline{\Phi}}}
\def\oY{{\overline{Y}}}
\def\oE{{\overline{E}}}

\newcommand{\IA}{{\mathbb A}}

\newcommand{\IG}{{\mathbb G}}

\newcommand{\Ik}{{\mathbbm k}}

\newcommand{\IN}{{\mathbb N}}
 
\newcommand{\IP}{{\mathbb P}} 
\newcommand{\IQ}{{\mathbb Q}} 
\newcommand{\IR}{{\mathbb R}}

\newcommand{\IT}{{\mathbb T}}

\newcommand{\IZ}{{\mathbb Z}}

\newcommand{\CD}{{\mathcal D}} 
\newcommand{\CE}{{\mathcal E}}
\newcommand{\CF}{{\mathcal F}}
\newcommand{\CG}{{\mathcal G}} 
\newcommand{\CH}{{\mathcal H}}
\newcommand{\CI}{{\mathcal I}}

\newcommand{\CL}{{\mathcal L}}

\newcommand{\CO}{{\mathcal O}}

\newcommand{\CW}{{\mathcal W}}
\newcommand{\CX}{{\mathcal X}}
\newcommand{\CY}{{\mathcal Y}}

\newcommand{\fa}{\mathfrak{a}}

\newcommand{\fm}{\mathfrak{m}}

\newcommand{\seq}{\subseteq}
\newcommand{\la}{\langle}
\newcommand{\ra}{\rangle}

\newcommand{\bu}{\bullet}
\newcommand{\lam}{\lambda}
\newcommand{\D}{\Delta}

\newcommand{\vep}{\varepsilon}

\title{Generalized optimal degenerations of Fano varieties}

\author{Linsheng Wang}
\address{Shanghai Center for Mathematical Sciences, Fudan University, Shanghai, 200438, China}
\curraddr{}
\email{linsheng\_wang@fudan.edu.cn}
\thanks{}
\date{}
\dedicatory{}

\begin{document}

\maketitle

\begin{abstract}
We prove a generalization of the algebraic version of Tian conjecture. Precisely, for any smooth strictly increasing function $g:\mathbb{R}\to\mathbb{R}_{>0}$ with ${\rm log}\circ g$ convex, we define the $\mathbf{H}^g$-invariant on a Fano variety $X$ generalizing the $\mathbf{H}$-invariant introduced by Tian-Zhang-Zhang-Zhu, and show that $\mathbf{H}^g$ admits a unique minimizer. Such a minimizer will induce the $g$-optimal degeneration of the Fano variety $X$, whose limit space admits a $g'$-soliton. We present an example of Fano threefold which has the same $g$-optimal degenerations for any $g$. 
\end{abstract}

%\tableofcontents

\section{Introduction}
As predicted by \cite[Conjecture 9.1]{Tia97}, a normalized K\"ahler-Ricci flow $\omega_t$ on a Fano manifold $M$ will converge in the Cheeger-Gromov-Hausdorff topology to $(M_\infty,\omega_\infty)$ with mild singularities, where $\omega_\infty$ is a K\"ahler-Einstein metric or a K\"ahler-Ricci soliton on the smooth part of $M_\infty$. 
This conjecture was widely studied, and has been solved now, see \cite{TZ16, Bam18, CW20, WZ21}. The limit $M_\infty$ is called the {\it optimal degeneration} of the Fano manifold $M$. 

There is an algebraic version of the above conjecture splitting the optimal degeneration into two steps, that is, ``semistable degeneration'' and ``polystable degeneration'', which was solved by \cite{HL23,HL20,BLXZ23}. It focus on the minimization problem of the so-called $\BH$-invariants $h(X,\D)$ (Definition \ref{Definition: H^g invariant}). The $\BH$-invariant was first introduced by \cite{TZZZ13} for holomorphic vector fields in the study of K\"ahler-Ricci flow on Fano manifolds. It was generalized to any special $\IR$-test configuration and any filtration by \cite{DS20} and \cite{HL20} respectively. Then the optimal degeneration problem can be studied in an algebraic geometry manner. 
By \cite[Theorem 3.15]{BLXZ23} and \cite[Theorem 4.9]{HL20}, for any log Fano pair $(X,\D)$, there exists a unique quasi-monomial valuation $v_0$ minimizing the $\BH$-invariant $h(X,\D)$. It was proved by \cite[Corollary 5.7]{BLXZ23} that the associated graded ring $\gr_{v_0}R$ is finitely generated. Hence it induces a (multistep) special degeneration of $(X,\D)$ to some weighted K-semistable log Fano triple $(X_0,\D_0,\xi_0)$ (where $\xi_0$ is a vector field determined by $v_0$, see Remark \ref{Remark. log Fano triple induced by special valuation}). More precisely, the log Fano pair $(X_0,\D_0)$ is $e^{-\la\sim,\xi_0\ra}$-weighted K-semistable, where $\alpha \mapsto e^{-\la\alpha,\xi_0\ra}$ is a function defined on the moment polytope $\BP$ of $(X_0,\D_0)$ with respect to some torus action. Moreover, $(X_0,\D_0,\xi_0)$ will specially degenerate to a unique weighted K-polystable log Fano triple $(Y,\D_Y,\xi_0)$ by \cite[Theorem 1.3]{HL20}. Finally, it was proved by \cite[Theorem 1.3]{BLXZ23} that $(Y,\D_Y,\xi_0)$ admits a K\"ahler-Ricci soliton based on \cite[Theorem 1.7]{HL23}. 

%Rather than K\"ahler-Ricci soliton, the $g_0$-soliton metrics plays a central roal in 

In Han-Li's and Blum-Liu-Xu-Zhuang's proof of the second step of the conjecture, that is, the ``polystable degeneration'' step, they worked not only on K\"ahler-Ricci solitons, but also on $g_0$-solitons. Precisely, they showed that for any smooth weight function $g_0:\BP\to \IR_{>0}$ (\ref{Eqnarray. weight function}), any $g_0$-weighted K-semistable log Fano pair $(X,\D)$ will specially degenerate to a $g_0$-weighted K-polystable log Fano pair $(Y,\D_Y)$, which is $g_0$-weighted reduced uniformly K-stable by \cite[Theorem 1.3]{BLXZ23}, hence admits a $g_0$-soliton by \cite[Theorem 1.7]{HL23}. Motivated by these works, one may ask whether there is an associated ``semistable degeneration'' step in the algebraic version of Tian conjecture or not. %More ambitiously, one may ask whether there is a {\it $g$-flow} on Fano manifolds and an associated Tian conjecture.  

In this paper, we give a generalization of the $\BH$-invariant, namely, the $\BH^g$-invariant for some
\begin{eqnarray}
\label{Eqnarray: function g}
\textrm{smooth strictly increasing function $g:\IR\to\IR_{>0}$ with $\log\circ g$ convex. }
\end{eqnarray}
This will lead to the ``semistable degeneration'' step asked in the previous paragraph. We aim to prove the following generalized version of Tian conjecture. 

\begin{thm}[Generalized Tian conjecture] 
\label{Theorem: Generalized Tian Conjecture}
Let $(X,\D)$ be a log Fano pair, and $g:\IR\to\IR_{>0}$ be a smooth strictly increasing function with $\log\circ g$ convex. Then the $\BH^g$-invariant (Definition \ref{Definition: H^g invariant}) of $(X,\D)$ admits a unique minimizer $v_0$, which is a special valuation (Theorem \ref{Theorem: special valuation, higher rank f.g.}), such that the (multistep) special degeneration $(X_0,\D_{0},\xi_0)$ induced by $v_0$ is $g'$-weighted K-semistable (precisely, $(X_0,\D_0)$ is $g'(-\la \sim,\xi_0\ra)$-weighted K-semistable).  
Moreover $(X_0,\D_{0},\xi_0)$ has a unique $g'$-weighted K-polystable special degeneration $(Y,\D_{Y}, \xi_0)$, which admits a $g'$-soliton metric. 
\end{thm}

%A valuation $v$ over $(X,\D)$ is called {\it special} if it is quasi-monomial, the associated graded ring $\gr_v R$ is finitely generated and inducing a (multistep) special degeneration with klt central fiber, see \cite[Theorem 4.1]{XZ22}. 
We say that $(Y,\D_{Y}, \xi_0)$ is the {\it $g$-optimal degeneration} of $(X,\D)$. 
The last statement of the theorem has been established by \cite{BLXZ23,HL20}. We aim to prove the first part of the theorem. 

\begin{rmk}\rm 
In the setting of $g$-optimal degenerations, the correct weighted stability notion is the $g'$-weighted K-stability, where $g'$ is the first order derivative of the function $g$. See Lemma \ref{Lemma: gHT: H-minimizer v_0 and delta^(g,v_0)} and Theorem \ref{Theorem: gHT: Weighted K-stability} for details. If we choose $g(x)=e^x$, then it reveals the ordinary optimal degeneration. In this case $g'(x)=g(x)$. 
\end{rmk}

The following theorem is an analog of \cite[Theorem 5.3]{HL20}, which is the key ingredient in finding $g$-optimal degenerations. 

\begin{thm}[Theorem \ref{Theorem: gHT: Weighted K-stability}]
\label{Theorem: intro. gHT: Weighted K-stability}
Let $v_0$ be a quasi-monomial valuation over $X$ with finitely generated associated graded ring $\gr_{v_0}R$, which induces a klt (multistep) special degeneration $(X_0,\D_{0},\xi_0)$. Then $v_0$ minimizes $\BH^g$ if and only if $(X_0,\D_{0},\xi_0)$ is $g'$-weighted K-semistable. 
\end{thm}

If Theorem \ref{Theorem: Generalized Tian Conjecture} is established, then it is natural to ask what is the relationship between the $g$-optimal degenerations of a log Fano pair $(X,\D)$ for different functions $g$. 

\begin{ques}
\label{Question: same g -optimal degenerations?}
Let $(X,\D)$ be a log Fano pair and $g,\bg$ be functions satisfying (\ref{Eqnarray: function g}). Let $(Y,\D_{Y}, \xi_0)$, $(\bY,\D_{\bY}, \bxi_0)$ be the $g$-, $\bg$-optimal degenerations of $(X,\D)$ respectively. When do we have 
\begin{eqnarray}
\label{Eqnarray: g- barg- optimal degeneration iso.}
(Y,\D_{Y}) \cong (\bY,\D_{\bY})? 
\end{eqnarray}
\end{ques}

If $(X,\D)$ is a toric log Fano pair, then the isomorphism (\ref{Eqnarray: g- barg- optimal degeneration iso.}) always holds since $(X,\D)$ is $g_0$-weighted K-polystable for any weight function $g_0:\BP\to \IR_{>0}$ (see Corollary \ref{Corollary: stable under g -optimal degeneration 2} for details). We have the following non-trivial examples given by \cite[Example 5.5 and 5.7]{Wang24}. 

\begin{thm}
For any Fano threefold in families №2.28, №3.14 and №2.23(a) of Mori-Mukai's list, the isomorphism (\ref{Eqnarray: g- barg- optimal degeneration iso.}) always holds. 
\end{thm}

\begin{rmk}\rm
Delcroix \cite{Del24} introduced the notion of {\it weight insensitive} Fano varieties, which are $g_0$-weighted K-semistable for any weight function $g_0$. 
The examples above are weight insensitive. He founded new examples of weight insensitive Fano threefolds and presented a weight sensitive example, for which Question \ref{Question: same g -optimal degenerations?} would have a negative answer. 
\end{rmk}

%It is natural to ask how about Fano threefolds in family №2.23(b) (\cite{MW24} or \cite[Example 5.9]{Wang24}). 

The paper is organized as follows. In Section \ref{Section: Preliminaries} we recall some basic notions in K-stability theory that we will use. We define the generalized $\BH$-invariant $\BH^g$ for log Fano pairs $(X,\D)$ in Section \ref{Section: Generalized H-invariants} and study the basic properties of it. In Section \ref{Section: Existence of H^g-minimizers and finite generation}, we show the existence of the $\BH^g$-minimizer and its finite generation property in the log Fano case. 
Finally, we give some examples of $g$-optimal degenerations in Section \ref{Section: Examples}.

%\noindent {\bf Acknowledgments}. The author would like to thank his advisor Gang Tian for constant support and guidance. He thanks Jiyuan Han for many helpful comments in this paper. He thanks Thibaut Delcroix for informing me to consider the $\IG$-equivariant $g$-optimal degenerations. He thanks Minghao Miao, Lu Qi, Kewei Zhang and Shengxuan Zhou for helpful discussions. He also thank the anonymous referees for their valuable comments and suggestions. The author was partially supported by the NKRD Program of China (\#2023YFA1010600), (\#2020YFA0713200) and LNMS. 

\noindent \textbf{Acknowledgments}. The author would like to express his gratitude to Gang Tian for his constant support and guidance. He is grateful to Jiyuan Han for helpful comments, to Thibaut Delcroix for the suggestion to consider $G$-equivariant $g$-optimal degenerations, and to Minghao Miao, Lu Qi, Kewei Zhang, and Shengxuan Zhou for helpful discussions. He also thanks the anonymous referees for their valuable comments and suggestions. The author was partially supported by the NKRD Program of China (\#2023YFA1010600), (\#2020YFA0713200) and LNMS.

\section{Preliminaries}
\label{Section: Preliminaries}
We work over an algebraically closed field $\Ik$ of characteristic $0$. A {\it pair} $(X, \D)$ consists of a normal variety $X$ and an effective $\IQ$-divisor $\D$ on $X$ such that $K_X+\D$ is $\IQ$-Cartier. 
A polarized klt pair $(X,\D;L)$ consists of a projective klt pair $(X,\D)$ and an ample $\IQ$-Cartier $\IQ$-divisor $L$ on $X$. It is called {\it log Fano} if $(X,\D)$ is klt and $L=-(K_X+\D)$. Fix an integer $l_0 > 0$ such that $l_0 L$ is Cartier. We denote by $R:=R(X;L):=\oplus_{m\in l_0\IN} R_m$ the section ring of $L$ where $R_m := H^0(X, mL)$. A projective pair $(X,\D)$ is called of {\it Fano type} if there exists a $\IQ$-divisor $\D'\ge \D$ such that $(X,\D')$ is a log Fano pair.  
%For the standard notions of {\it filtrations, valuations} and {\it test configurations} of a polarized pair $(X,\D;L)$, we refer the reader to \cite{Xu24}. 

A $\IQ$-divisor $0\le \Gamma\sim_{\IQ} -(K_X+\D)$ is called a {\it $\IQ$-complement} of a log Fano pair $(X,\D)$ if $(X,\D+\Gamma)$ is log canonical. 
A {\it model} $\pi: (Y,E)\to (X,\D)$ consists of a proper birational morphism $\pi : Y\to X$ and a reduced divisor $E$ on $Y$. It is called a {\it log smooth model} if $(Y, \Supp(E+\Ex(\pi)+\pi^{-1}_*\D))$ is simple normal crossing. 
It is called a {\it toroidal model} if it is locally quotient of simple normal crossing pairs by abelian groups.  
If $E=E_1+\cdots +E_r$, then the closed subsets $\cap_{i\in J} E_i \seq Y$ ($J\seq \{1,\cdots,r\}$) are called {\it strata} of $(Y,E)$.

\subsection{Valuations}

Let $K$ be a field. An $\IR$-{\it valuation} $v$ on $K$ is a function $v: K^\times \to \IR$ such that $v(fg)=v(f) + v(g), v(f+g)\ge \min\{v(f), v(g)\}$ for all $f, g \in K^\times$. For convenience, we set $v(0)=+\infty$. The {\it trivial valuation} $v_{\triv}$ is defined as $v_{\triv}(f)=0$ for all $f\in K^\times$. Let $X$ be a normal variety. A {\it valuation} $v$ on $X$ is an $\IR$-valuation on the rational function field $K(X)$ with a center on $X$ and $v|_{\Ik^\times}=0$. Recall that the {\it center} of $v$, denoted by $c_X(v)$, is a scheme-theoretic point $\zeta$ on $X$ such that $v\ge 0$ on $\CO_{X,\zeta}$ and $v>0$ on the maximal ideal $\fm_{\zeta}\seq \CO_{X,\zeta}$. We denote by $C_X(v)=\overline{c_X(v)}\subseteq X$ the corresponding closed irreducible subscheme on $X$. If $X$ is proper, then every valuation $v$ has a unique center on $X$. For any valuation $v$ on $X$, we can define the {\it log discrepancy} $A_{X,\D}(v)$ by \cite[Section 5]{JM12}. We denote by  $\Val_{X}^\circ$ the subset of non-trivial valuations $v\in\Val_{X}$ with $A_{X,\D}(v)<+\infty$ of $v$ for some $\IQ$-divisor $\D$ such that $K_X+\D$ is $\IQ$-Cartier (it does not depend on the choice of $\D$).

Let $(Y,E=E_1+\cdots+E_r)\to X$ be a log smooth model and $\alpha=(\alpha_1,\cdots, \alpha_r) \in \IR_{\ge 0}^r$, we define the valuation $v_\alpha$ by 
\begin{eqnarray}
\label{Eqnarray. quasi-monomial valuations}
v_\alpha(f) = \min\{ \sum_{1\le i\le r} \alpha_i \beta_i\mid f_\beta\ne 0 \},  
\end{eqnarray}
for any $f = \sum_\beta f_\beta z^\beta \in \hat{\CO}_{X,\eta}$, where $\eta$ is the generic point of some irreducible component of $\cap_{1\le i\le r} E_i$ and $z_1,\cdots, z_r \in \hat{\CO}_{X,\eta}$ are local functions such that $E_i = (z_i=0)$. The valuation $v_\alpha$ is called a {\it quasi-monomial valuation} over $X$. We denote by 
\begin{eqnarray}
\QM_\eta(Y,E) = \{v_\alpha\in \Val_{X}\mid \alpha \in \IR_{\ge 0}^r \} \seq \Val_{X}, 
\end{eqnarray}
which is isomorphic to the simplicial cone $\IR^r_{\ge 0}$, and we say that $\QM_\eta(Y,E)$ is a {\it quasi-monomial simplicial cone} over $X$. We denote by $\QM(Y,E) = \cup_\eta \QM_\eta(Y,E)$, where $\cup_\eta$ is taken over all the generic points of strata of $(Y,E)$.  

Let $(X,\D)$ be a log canonical pair. A valuation $v$ over $X$ is called a {\it log canonical place} of $(X,\D)$ if $A_{X,\D}(v) = 0$. We denote by $\LC(X,\D)\seq \Val_{X}$ the subset of all the log canonical places of $(X,\D)$. If $\pi: (Y,E)\to (X,\D)$ is a log smooth model, then 
$$\LC(X,\D)\seq \QM(Y, \Supp(E+\Ex(\pi)+\pi^{-1}_*\D)). $$

\subsection{Filtrations}

Let $(X,\D;L)$ be a polarized klt pair of dimension $n$. Fix $\l0\in\IZ_{>0}$ such that $\l0 L$ is Cartier. Following \cite[2.1]{BJ20}, a {\it graded linear series} $V_\bu=\{V_m\}_{m\in\l0\IN}$ of $L$ is a sequence of subspaces $V_m\seq R_m$ such that $V_0=\Ik$ and $V_m\cdot V_{m'}\seq V_{m+m'}$. %Equivalently, we may consider $V_\bu = \oplus_{m\in\l0\IN}V_m \seq R_\bu$ as a graded sub-ring. 
We assume that $V_\bu$ {\it contains an ample series}, that is, $H^0(X,mA)\seq V_m$ for $m\gg0$, where $A$ is an ample $\IQ$-divisor such that $|L-A|_\IQ\ne \varnothing$. In this case, the volume $\vol(V_\bu)$ is well-defined by 
\begin{eqnarray*} 
\vol(V_\bu) = \mathop{\lim}_{m\to \infty} \frac{\dim V_m}{m^n/n!}>0.
\end{eqnarray*}
%For such a graded linear series $V_\bu$, we may construct a convex body $\BO=\BO(V_\bu)\seq \IR^n$ called the {\it Okounkov body} by choosing an admissible flag on $X$, such that $\vol(\BO(V_\bu))=\frac{1}{n!}\vol(V_\bu)$. See for example \cite{LM09}. 
Note that the section ring $R_\bu=R(X;L)$ is a graded linear series containing an ample series. 

\begin{defi}\rm 
A {\it filtration} $\CF$ on $V_\bu$ is a collection of subspaces $\CF^\lam V_m \subseteq V_m$ for each $\lam \in \IR$ and $m\ge 0$ such that
\begin{itemize}
\item {\it Decreasing.} $\CF^\lam V_m \supseteq \CF^{\lam'}V_m $ for  $\lam \le \lam'$; 
\item {\it Left-continuous.} $\CF^\lam V_m=\CF^{\lam-\epsilon}V_m$ for $0<\epsilon \ll 1$; 
\item {\it Linearly bounded.} there exists $C>0$ such that $\CF^{-mC}V_m=V_m$ and $\CF^{mC}V_m=0$ for all $m>0$; 
\item {\it Multiplicative.} $\CF^\lam V_m \cdot \CF^{\lam'}V_{m'} \subseteq \CF^{\lam+\lam'}V_{m+m'}$. 
\end{itemize}
For any $s\in V_m$, we set 
$\ord_\CF(s)=\max\{\lam\mid s\in\CF^{\lam}V_m\}.$
A basis $\{s_1,\cdots, s_{N_m}\}$ of $V_m$ is called {\it compatible} with $\CF$ if $\CF^\lam V_m$ is generated by $\{s_i\mid \ord_\CF(s_i)\ge \lam \}$ for any $\lam \in \IR$. 
In this case, the set 
\begin{eqnarray*} 
\Gamma_m(\CF)=\Gamma_m(V_\bu;\CF) := \{\ord_\CF(s_j)\mid 1\le j\le N_m\}
\end{eqnarray*}
is equal to the set of numbers $\lam\in\IR$ where $\CF^\lam V_m$ jumps, hence
is independent of the choice of the compatible basis $\{s_j\}$. It is called the {\it successive minima} of $\CF$ along $V_m$. 
Since $\CF$ is multiplicative, we see that the subset of $\IR\times \l0\IN$
\begin{eqnarray*} 
\Gamma(\CF) = \Gamma(V_\bu;\CF) := 
\{(\lam,m)\in \IR\times \l0\IN \mid \lam\in\Gamma_m(\CF)\},
\end{eqnarray*}
is a monoid (semigroup with unit). In particular, $(0,0)\in \Gamma(\CF)$, that is, $\CF^0V_0 =\Ik$ and $\CF^{\vep}V_0=0$ for any $\vep>0$. Let $\tGamma(\CF)\seq \IR\times \l0 \IZ$ be the (free abelian) subgroup generated by $\Gamma(\CF)$. 
Then the rank of $\tGamma(\CF)$ is called the {\it rational rank} of $\CF$. 
%We also denote by $$\tGamma_m(\CF) = \tGamma(\CF) \cap (\IR\times \{m\}) \supseteq \Gamma_m(\CF).$$
Since $\CF$ is linearly bounded, the sequence of numbers 
$$\lam^{(m)}_\max = \max\{\lam\in \IR\mid \CF^{\lam}V_m \ne 0\} = \max\,\Gamma_m(V_\bu;\CF) $$ 
is also linearly bounded, that is, 
\begin{eqnarray*} 
\lam_\max(V_\bu; \CF) 
:= \mathop{\sup}_{m\in \IN} \frac{\lam^{(m)}_\max}{m} 
= \mathop{\lim}_{m\rightarrow \infty} \frac{\lam^{(m)}_\max}{m} < +\infty. 
\end{eqnarray*}
By \cite[Theorem 5.3]{BHJ17}, we have 
\begin{eqnarray*} 
\lam_\max(V_\bu; \CF) 
= 
\sup\{t\in\IR\mid \vol(\CF^{(t)}V_\bu) = 0\}. 
\end{eqnarray*}
As in \cite[Corollary 5.4]{BHJ17}, we define 
\begin{eqnarray*} 
\lam_\min(V_\bu; \CF) 
:= 
\inf\{t\in\IR\mid \vol(\CF^{(t)}V_\bu) < \vol(V_\bu)\}. 
\end{eqnarray*}
\end{defi}

For example, if $v$ is a valuation over $X$, then $\CF_v^\lam V_m := \{s\in V_m\mid v(s)\ge \lam\}$ defines a filtration on $V_\bu$. It is linearly bounded if $A_{X,\D}(v)<+\infty$ by \cite[Lemma 3.1]{BJ20}, which holds for quasi-monomial valuations over $X$. 

For any filtration $\CF$ on $V_\bu$ and $a\in\IR_{>0}, b\in \IR$, we define the $a$-{\it rescaling} and $b$-{\it shift} of $\CF$ by
\begin{eqnarray*} 
(a\CF)^\lam V_m := \CF^{\lam/a}V_m, \quad
  \CF(b)^\lam V_m := \CF^{\lam-bm} V_m, 
\end{eqnarray*}
and we also denote by $a\CF(b):=(a\CF)(b)$, that is $(a\CF(b))^\lam V_m = \CF^{\frac{\lam - bm}{a}} V_m$. 
%If $V_\bu$ admits a rank $r$ torus action, equivalently, $V_\bu$ admits a weight decomposition $V_{m}=\oplus_{\alpha \in M} V_{m,\alpha}$, where $M\cong \IZ^r$ and $V_{m,\alpha}\cdot V_{m',\alpha'}\seq V_{m+m',\alpha+\alpha'}$, then we may define the $\xi$-twist. 
We denote by 
\begin{eqnarray} 
\label{Eqnarray. Def. gr_F V}
\gr_\CF^\lam V_m 
:= \CF^{\lam}V_m / \CF^{> \lam} V_m, \quad 
\gr_\CF V_m 
:= \bigoplus_{\lam\in \IR} \gr_\CF^\lam V_m 
= \bigoplus_{\lam\in \Gamma_m(\CF)} \gr_\CF^\lam V_m, 
\end{eqnarray}
where $\CF^{>\lam}V_m := \cup_{\mu>\lam} \CF^{\mu}V_m$. The last equality holds since $\gr_\CF^\lam V_m = 0$ if $\lam \notin \Gamma_m(\CF)$.

\subsection{Okounkov bodies, concave transforms and DH measures}

\begin{defi}\rm
Let $\CF$ be a linearly bounded filtration on $V_\bu$. By \cite[Lemma 3.17]{Xu24}, for any $t < \lam_\max(V_\bu;\CF)$, the graded linear subseries $\CF^{(t)} V_\bu \seq V_\bu$ defined by $(\CF^{(t)} V)_{m}=\CF^{mt}V_{m}$ contains an ample series. We denote the Okounkov body (\cite{LM09,BJ20}) of $\CF^{(t)} V_\bu$ by $\BO^{(t)}$, and let $\BO=\BO(V_\bu)$. Then $\BO^{(t)}\seq \BO$ is a descending collection of convex bodies. The {\it concave transform} of $\CF$ is the function on $\IR^n$ defined by 
\begin{eqnarray*} 
G_\CF(y)=\sup\{t\in \IR\mid y\in\BO^{(t)}\}.
\end{eqnarray*}
Note that $G_\CF$ is concave and upper-semicontinuous. The linear boundedness of $\CF$ guarantees that $\BO^{(-C)}=\BO$ and $\BO^{(C)}=0$. In other word, $\BO$ is contained in the level set $\{-C\le G_\CF\le C\}\seq \IR^n$. 
\end{defi}

\begin{lem}
For any $a\in\IR_{>0}, b\in \IR$, we have $G_{a\CF(b)}=aG_\CF+b$. 
\end{lem}

We define
\begin{eqnarray*} 
\lam_\min(V_\bu; \CF) 
:= \inf\{t\in\IR\mid \vol(\CF^{(t)}V_\bu) < \vol(V_\bu)\}. 
\end{eqnarray*}

\begin{defi}\rm
Let $\CF$ be a linearly bounded filtration on $V_\bu$. Then for any $m\in\lN$, the bounded function $t\mapsto \dim \CF^{mt}V_m$ is non-increasing on $\IR$. By taking derivative in the Lebesgue-Stieltjes sense, we get the following discrete measure on $\IR$: 
\begin{eqnarray*} 
\DH_{\CF,m} 
:= -\frac{\dif}{\dif t} \frac{\dim \CF^{mt} V_m}{\dim V_m}
= \sum_{t\in\IR} \delta_{t} \cdot \frac{\dim \gr_\CF^{mt} V_m}{\dim V_m} 
= \sum_{mt\in \Gamma_m(\CF)} \delta_{t} \cdot \frac{\dim \gr_\CF^{mt} V_m}{\dim V_m}, 
\end{eqnarray*}
where $\delta_{t}$ is the Dirac measure at $t\in \IR$. By \cite{BC11,BHJ17}, $\DH_{\CF,m}\to \DH_\CF$ converges weakly as $m\to \infty$, where 
\begin{eqnarray*} 
\DH_\CF
= -\frac{\dif}{\dif t} \frac{\vol(\CF^{(t)}V_\bu)}{\vol(V_\bu)}
\end{eqnarray*}
is called the {\it Duistermaat-Heckman (DH) measure} of $\CF$. 

Let $\CG$ be another linearly bounded filtration on $V_\bu$. Then for any $m\in\lN$, the bounded function 
\begin{eqnarray*} 
f_m(x,y) = \dim \CF^{mx}V_m\cap\CG^{my} V_m
\end{eqnarray*}
is non-increasing on $\IR^2$ in the Hardy-Krause sense (hence has bounded variation), that is, 
\begin{eqnarray*} 
f_m(x_1,y_1) - f_m(x_1,y_2) - f_m(x_2,y_1) + f_m(x_2,y_2) \le 0, 
\end{eqnarray*}
for any $x_1\le x_2$ and $y_1\le y_2$ since 
\begin{eqnarray*} 
(\CF^{mx_1}V_m\cap\CG^{my_2} V_m) \cap 
(\CF^{mx_2}V_m\cap\CG^{my_1} V_m) =\CF^{mx_2}V_m\cap\CG^{my_2} V_m, \\ 
(\CF^{mx_1}V_m\cap\CG^{my_2} V_m) +
(\CF^{mx_2}V_m\cap\CG^{my_1} V_m) \seq 
\CF^{mx_1}V_m\cap\CG^{my_1} V_m.
\end{eqnarray*}
Taking derivative, we get the following discrete measure on $\IR^2$ (see also \cite[3.1.3]{BLXZ23}):
\begin{eqnarray*} 
\DH_{\CF,\CG,m} 
:= -\frac{\partial^2}{\partial x \partial y} \frac{\dim \CF^{mx}V_m\cap\CG^{my} V_m}{\dim V_m}
= \sum_{(x,y)\in\IR^2} \delta_{(x,y)} \cdot \frac{\dim \gr_\CF^{mx}\gr_\CG^{my} V_m}{\dim V_m}
\end{eqnarray*}
which also converges weakly to 
\begin{eqnarray*} 
\DH_{\CF,\CG} 
= -\frac{\partial^2}{\partial x \partial y} \frac{\vol(\CF^{(x)}\CG^{(y)} V_\bu)}{\vol(V_\bu)} 
\end{eqnarray*}
as $m\to \infty$ by \cite[Theorem 3.3]{BLXZ23}, where $\CF^{(x)}\CG^{(y)} V_\bu$ is the graded linear series defined by 
\begin{eqnarray*} 
(\CF^{(x)}\CG^{(y)} V_\bu)_m := \CF^{mx}V_m\cap\CG^{my} V_m.
\end{eqnarray*}
This measure is called the {\it DH measure compatible with both $\CF$ and $\CG$}. 
\end{defi}

The two measures defined above both have compact support since $\CF$ and $\CG$ are linearly bounded. 
Let $f$ be a continuous function on $\IR$, then
\begin{eqnarray*} 
\int_{\IR^2} f(x) \DH_{\CF,\CG}(\dif x \dif y) = \int_\IR f(x) \DH_{\CF}(\dif x). 
\end{eqnarray*}
By \cite[2.5]{BJ20}, we also have 
\begin{eqnarray*}  
\DH_\CF = G_{\CF,*} \LE,
\end{eqnarray*}
where $\LE$ is the Lebesgue measure on the Okounkov body $\BO=\BO(V_\bu)$. 

We define the {\it $L^1$-distance} of $\CF$ and $\CG$ by 
\begin{eqnarray*}  
d_1(\CF,\CG) := \int_{\IR^2} |x-y| \DH_{\CF,\CG}(\dif x \dif y), 
\end{eqnarray*}
and say that $\CF,\CG$ are {\it equivalent} if $d_1(\CF,\CG)=0$. Let $v,w$ be valuations over $X$, if $\CF_v$ and $\CF_w$ are equivalent, then $v=w$ by \cite[Proposition 2.27]{HL20}, see also \cite[Lemma 3.16]{BLXZ23}.

\subsection{Log canonical slopes and $\BL$-functionals}
We first recall the log canonical thresholds of graded ideal sequences, see for example \cite[Section 1.2.2]{Xu24}. 
Let $(X,\D)$ be a log canonical pair and $\fa\seq \CO_X$ be a non-zero ideal sheaf. For any valuation $v$ over $X$, recall that $v(\fa) = \min\{v(f)\mid f\in \fa\}$. The {\it log canonical threshold of $\fa$ with exponent $c>0$} is defined by 
$$\lct(X,\D;\fa^c):= \inf_v \frac{A_{X,\D}(v)}{c\cdot v(\fa)}. $$
Let $\fa_\bu = \{a_k \}_{k\in \IN}$ be a graded sequence of ideal sheaves $\fa_k\seq \CO_X$, that is, $\fa_0 = \CO_X, \fa_k\supseteq \fa_{k+1}, \fa_k\cdot \fa_l \seq \fa_{k+l}$. Then for any valuation $v$ over $X$, by \cite[Lemma 1.46]{Xu24}, we have 
$$v(\fa_\bu) := \mathop{\lim}_{k\to \infty} \frac{v(\fa_k)}{k} = \mathop{\inf}_{k\in \IN} \frac{v(\fa_k)}{k}. $$
The {\it log canonical threshold of $\fa_\bu$ with exponent $c>0$} is defined by 
$$\lct(X,\D;\fa_\bu^c):= \inf_v \frac{A_{X,\D}(v)}{c\cdot v(\fa_\bu)}. $$

\begin{defi}\rm 
\label{Definition: log canonical slope}
Let $(X,\D;L)$ be a polarized klt pair and $\CF$ be a linearly bounded filtration on $R=R(X;L)$. The {\it base ideal sequence} $I^{(t)}_\bu = \{I_{m,mt}\}_{m\in l_0\IN}$ of $\CF$ is defined by 
\begin{eqnarray*}
I_{m,mt}
\,\,\,=\,\,\,
I_{m,mt}(L;\CF)
\,\,\,:=\,\,\,{\rm im}
\Big(\CF^{mt}H^0(X,mL)\otimes \CO(-mL)\to \CO_X\Big), 
\end{eqnarray*}
for any $m\in l_0\IN$ and $t\in\IR$. The {\it log canonical slope} of $\CF$ is defined by
\begin{eqnarray*}
\mu(\CF)
\,\,\,=\,\,\, 
\mu_{X,\D;L}(\CF)
\,\,\,:=\,\,\,
\sup\Big\{
t\mid \lct(X,\D;I^{(t)}_\bu)\ge 1
\Big\}. 
\end{eqnarray*}
Note that $I^{(t)}_\bu=0$ (hence $\lct(X,\D;I^{(t)}_\bu)=0$) when $t>\lam_\max$. We have $\mu(\CF)\le \lam_\max$. 
\end{defi}

\begin{lem}
For any $a\in\IR_{>0}, b\in \IR$, we have $\mu(a\CF(b))=a\mu(\CF)+b$. 
\end{lem}

Consider the following function of $t\in \IR$ in the definition of $\mu(\CF)$, 
\begin{eqnarray*}
f(t)
=\lct(X,\D; I^{(t)}_\bu)
=\inf_v\frac{A_{X,\D}(v)}{v(I^{(t)}_\bu)}, 
\end{eqnarray*}
where the infimum runs over all the valuations over $X$. We have the following useful lemma in computing log canonical slope. 

\begin{lem}\cite[Proposition 3.46]{Xu24}
\label{Lemma: log canonical thresholds(I^(t)) is continuous and strictly decreasing}
The function $f(t)$ is continuous non-increasing on $(-\infty, \lam_\max)$. If we set $\mu_{+\infty}= \sup\{t\mid \lct(X,\D; I^{(t)}_\bu)=+\infty\}$, then $f(t)$ is strictly decreasing on $[\mu_{+\infty}, \lam_\max)$. 
\end{lem}

As a consequence, we have 
\begin{eqnarray}
\label{Eqnarray: mu <= A}
\mu_{X,\D;L}(\CF_v) \le A_{X,\D}(v), 
\end{eqnarray}
for any valuation $v$ over $X$. 
Indeed, we only need to prove the inequality when $A_{X,\D}(v) < \lam_\max$ since $\mu(\CF_v)\le \lam_\max$. By definition, we have $v(I^{(t)}_\bu) \ge t$. Hence for any $t\ge  A_{X,\D}(v)$, we have $\lct(X,\D;I^{(t)}_\bu) \le \frac{A_{X,\D}(v)}{v(I^{(t)}_\bu)} \le 1$. So $\mu(\CF_v)\le A_{X,\D}(v)$ by Lemma \ref{Lemma: log canonical thresholds(I^(t)) is continuous and strictly decreasing}. 

\begin{lem}
\label{Lemma: mu=A}
If there exists $\Gamma\in |L|_\IQ$ such that $(X,\D+\Gamma)$ is log canonical, and $v$ is a log canonical place of $(X,\D+\Gamma)$. Then $\mu_{X,\D;L}(\CF_v) = A_{X,\D}(v)$. 
\end{lem}

\begin{proof}
Assume that $\Gamma\in \frac{1}{m}|mL|$. Since $v(\Gamma)=A_{X,\D}(v)$, we have $\Gamma\in \frac{1}{m}|\CF_v^{mA_{X,\D}(v)}R_m|$ and
\begin{eqnarray*}
\lct(X,\D;I^{(A_{X,\D}(v))}_\bu) 
\ge \lct(X,\D;\Gamma) 
\ge 1. 
\end{eqnarray*}
Hence $\mu(\CF_v)\ge A_{X,\D}(v)$. We conclude by (\ref{Eqnarray: mu <= A}). 
\end{proof}

\begin{defi}\rm 
\label{Definition: L-functional}
Let $\CF$ be a linearly bounded filtration on $R$, and $e_-, e_+ \in \IZ$ such that $\CF^{me_-}R_m = R_m$ and $\CF^{me_+}R_m =0$ for any $m\in l_0\IN$. Recall that $I_{m,\lam}$ is the base ideal sequence of $\CF$ (Definition \ref{Definition: log canonical slope}). We denote by 
\begin{eqnarray*}
\CI_m(e_+,e_-) &=& \CI_m(\CF; e_+, e_-) \\
&:=& 
I_{m,me_-} \cdot s^{-me_- +me_+} + 
I_{m,me_- + 1} \cdot s^{-(me_- + 1) +me_+} +  \cdots +
I_{m,me_+} \cdot s^{0} \seq \CO_X[s]. 
\end{eqnarray*}
Since $I_{m,me_-} = \CO_X, I_{m,me_+} = 0$ and $ \CO_X \cdot s^{-(me_- - 1)} \seq \CO_X \cdot s^{-me_-}$, we see that $\CI(e_+ + a, e_- - b) = \CI(e_+, e_-) s^{ma}$ for any $a,b \in \IN$. Hence $\CI_m(e_+) := \CI_m(e_+,e_-)$ is independent of the choice of $e_-$ and 
\begin{eqnarray*}
%\CI_m(\CF) =
\CI_m := \CI_m(e_+)\cdot s^{-me_+} \seq \CO_X[s,s^{-1}] 
\end{eqnarray*}
is independent of the choice of $e_+$. 
The {\it $\BL$-functional} of $\CF$ is defined by 
\begin{eqnarray*}
\BL(\CF) = \BL_{X,\D;L}(\CF)  := \mathop{\lim}_{m\to \infty} \lct(X_{\IA^1}, \D_{\IA^1} + \CI_m^\frac{1}{m};X_0) - 1, 
\end{eqnarray*}
where the limit exists by \cite[Lemma 1.49]{Xu24}. 
\end{defi}

\begin{lem}\cite[Theorem 3.52]{Xu24}
For any linearly bounded filtration $\CF$ on $R$, we have 
\begin{eqnarray*}  
\mu(\CF)=\BL(\CF).
\end{eqnarray*}
\end{lem}

%\begin{rmk}\rm The log canonical slope is used in the construction of moduli, and the $\BL$-functional is well-behaved on test configurations of Fano varieties. \end{rmk}

\subsection{Higher rank finite generation}
Let $(X,\D = \sum_i a_i \D_i)$ be a log Fano pair, where $\D_i$ are the irreducible components of $\D$. 
Assume that $v\in \Val_{X}^\circ$ is a quasi-monomial valuation such that the associated graded ring 
\begin{eqnarray}
\label{Eqnarray. Gr_v R}  
\gr_v R = \bigoplus_{m\in l_0\IN} \bigoplus_{\lam \in \IR_{\ge 0}} \gr_v^\lam R_m, \quad
\gr_v^\lam R_m = \CF_v^{\lam} R_m /\CF_v^{>\lam} R_m, 
\end{eqnarray}
is finitely generated. We denote by 
\begin{eqnarray*}  
X_v = \Proj \,\gr_v R,  \quad \D_v = \sum_i a_i \D_{i,v}, 
\end{eqnarray*}
where $\D_{i,v} \seq X_v$ is the corresponding degeneration divisor of $\D_i$ on $X_v$. More precisely, let $I_{\D_i}\seq R$ be the graded ideal of $\D_i$. We define the {\it initial term degeneration} $\Bin(I_{\D_i}) \seq \gr_v R$ of $I_{\D_i} \seq R$ by 
\begin{eqnarray*}  
\Bin(I_{\D_i}) = \bigoplus_{m\in l_0\IN, \lam\in \IR} I_{\D_i,m,\lam}, \quad 
I_{\D_i,m,\lam} = \{ \bar{s} \in \gr^\lam_v R_m \mid s\in I_{\D_i}, v(s)\ge \lam \}.
\end{eqnarray*}
Then $\D_{v,i}$ is the divisorial part of the sub-scheme $V(\Bin(I_{\D_i})) \seq X_v$ defined by $\Bin(I_{\D_i})$. In other words, $\D_{v,i}$ and $V(\Bin(I_{\D_i}))$ coincide away from a codimension $2$ subset of $X_v$. 

We say that the quasi-monomial valuation $v$ is a {\it special valuation} over $(X,\D)$ if $\gr_v R$ is finitely generated and $(X_v,\D_v)$ is klt. Recall that a projective birational model $\pi:(Y,E=E_1+\cdots+E_r)\to (X,\D)$ is of {\it qdlt Fano type} if there exists an effective $\IQ$-divisor $D$ on $Y$ such that $-(K_Y+D+E)$ is ample over $Z$, $(Y, D+E)$ is qdlt (i.e. quotient-dlt, see \cite[Definition 5.4]{Xu24}), $\lfloor D+E \rfloor = E$ and $D+E\ge \pi^{-1}_* \D$. 
We have the following deep theorem of higher rank finite generation developed by \cite{LXZ22,XZ22,Xu24}. 
\begin{thm}
\label{Theorem: special valuation, higher rank f.g.}

Let $(X,\D)$ be a log Fano pair, and $v$ be a quasi-monomial valuation over $X$. The following statements are all equivalent. 

\begin{enumerate}[{\rm \quad (a)}]
\item The associated graded ring $\gr_v R$ is finitely generated, and the central fiber $(X_v,\D_v)$ of the induced degeneration is klt. 

\item There exists a special $\IQ$-complement $\Gamma$ of $(X,\D)$ with respect to some toroidal model $\pi: (Y,E)\to (X,\D)$ such that $v\in \QM(Y,E)\cap \LC(X,\D+\Gamma)$. 

\item There exists a qdlt Fano type model $\pi: (Y,E)\to (X,\D)$ such that $v\in \QM(Y,E)$. 
\end{enumerate}
In this case, the valuation $v$ is called {\rm special} with respect to $(X,\D)$. 
\end{thm}

\begin{rmk} \rm 
\label{Remark. log Fano triple induced by special valuation}
Moreover, the valuation $v$ induces a co-weight vector $\xi_v$ on $(X_v,\D_v)$. Assume that $v$ is of rational rank $r$. Then there exists a log smooth model $(Y,E=E_1+\cdots+E_n)\to X$ with a closed point $p\in E_1\cap\cdots\cap E_n$ such that $v = v_\xi \in \QM_p(Y,E) \cong \IR^{n}_{\ge 0}$ for some $\xi = (\xi_1,\cdots,\xi_r,0,\cdots, 0)$, where $\xi_1,\cdots,\xi_r \in \IR_{>0}$ are $\IQ$-linearly independent. Then the sum $\oplus_\lam$ in (\ref{Eqnarray. Gr_v R}) is taken over the monoid: 
\begin{eqnarray*}  
\Gamma = \xi_1\IZ_{\ge 0} + \cdots + \xi_r\IZ_{\ge 0} \seq \IR_{\ge 0}. 
\end{eqnarray*}
We denote by 
\begin{eqnarray*}  
R_v = \bigoplus_{m\in l_0\IN} \bigoplus_{\alpha\in\IZ_{\ge 0}^r}  R_{v,m,\alpha}, \quad 
R_{v,m,\alpha} = \gr_v^{\alpha_1\xi_1+\cdots+\alpha_r\xi_r} R_m.  
\end{eqnarray*}
Then the filtration $\CF_v$ on $R$ decents to the $\IG_m^r$-invariant filtration $\CF_0$ on $R_v$: 
\begin{eqnarray*}  
\CF_0^\lam R_{v,m} = \bigoplus_{\alpha\in\IZ_{\ge 0}^r, \alpha_1\xi_1+\cdots+\alpha_r\xi_r \ge \lam}  R_{v,m,\alpha}, 
\end{eqnarray*}
which is just the filtration of the co-weight valuation $\wt_{\xi_v}$ determined by 
$$\xi_v = (\xi_1,\cdots,\xi_r) \in N_\IR(\IG_m^r) \cong \IR^r. $$
We say that $(X_v,\D_v,\xi_v)$ is the {\it (multistep) special degeneration} induced by the special valuation $v$ on $(X,\D)$. 
\end{rmk}

Motivated by \cite[Lemma 2.7]{LX18}, \cite[Lemma 3.3]{XZ22} and \cite[Lemma 4.2]{Che24}, we have the following characterization of weakly special valuations. See \cite[Theorem 3.11]{Wan25} for the local version. 

\begin{thm}
\label{Theorem: weakly special valuations}
Let $(X,\D)$ be a log Fano pair, and $v$ be a quasi-monomial valuation over $X$. The following statements are all equivalent. 
\begin{enumerate}[{\rm \quad (a)}]
\item $\mu(\CF_v) = A_{X,\D}(v)$. 
\item There exists a $\IQ$-complement $\Gamma$ of $(X,\D)$ such that $v\in\LC(X,\D+\Gamma)$. 
\end{enumerate}
In this case, the valuation $v$ is called {\rm weakly special} with respect to $(X,\D)$. 
\end{thm}

\begin{proof}
By Lemme \ref{Lemma: mu=A}, we have (b) $\Rightarrow$ (a). Now we prove (a) $\Rightarrow$ (b). 
By \cite{HMX14}, there exists $\vep>0$ depending only on $\dim X$ and coefficients of $\D$ such that, for any birational morphism $\pi:Y\dashrightarrow X$ and any reduced divisor $E$ on $Y$, the pair $(Y,\pi_*^{-1}\D+(1-\vep)E)$ is log canonical if and only if $(Y,\pi_*^{-1}\D+E)$ is.

Let $\mu=\mu(\CF_v) = A_{X,\D}(v)$. This is equivalent to $v$ computing $\lct(X,\D; I^{(\mu)}_\bu)=1$. Since $v$ is a quasi-monomial valuation over $X$, there exists a quasi-monomial simplicial cone $\sigma\seq \Val_X$ containing $v$. The functions $w\mapsto A_{X,\D}(w)$ and $w\mapsto w(I^{(\mu)}_\bu)$ are linear and concave on $\sigma$ respectively. Hence the function $A_{X,\D+I^{(\mu)}_\bu}(-): \sigma \to \IR$, 
\begin{eqnarray}
\label{Eqnarray: function A(-) on sigma}
w\mapsto A_{X,\D+I^{(\mu)}_\bu}(w) = A_{X,\D}(w) - w(I^{(\mu)}_\bu)
\end{eqnarray}
is convex on $\sigma$. In particular, it is Lipschitz on $\sigma$. Hence there exists a constant $C>0$ such that 
\begin{eqnarray*}
%A_{X,\D+I^{(\mu)}_\bu}(w) = 
|A_{X,\D+I^{(\mu)}_\bu}(w) -
   A_{X,\D+I^{(\mu)}_\bu}(v) |
\le C|w-v|. 
\end{eqnarray*}
On the other hand, $A_{X,\D+I^{(\mu)}_\bu}(w)\ge 0$ for any $w\in\sigma$ since $v$ compute $\lct(X,\D; I^{(\mu)}_\bu)=1$. Hence
\begin{eqnarray}
0 \le 
A_{X,\D+I^{(\mu)}_\bu}(w) = 
|A_{X,\D+I^{(\mu)}_\bu}(w) -
   A_{X,\D+I^{(\mu)}_\bu}(v) |
\le C|w-v|. 
\end{eqnarray}
By Diophantine approximation \cite[Lemma 2.7]{LX18}, there exist divisorial valuations $v_1,\cdots, v_r$ and positive integers $q_1,\cdots,q_r,c_1,\cdots,c_r$ such that 
\begin{itemize}
\item $\{v_1,\cdots,v_r\}$ spans a quasi-monomial simplicial cone in $\Val_X$ containing $v$; 
\item for any $1\le i\le r$, there exists a prime divisor $E_i$ over $X$ such that $q_iv_i=c_i\ord_{E_i}$; 
\item $|v_i-v|< \frac{\vep}{2Cq_i}$ for any $1\le i\le r$. 
\end{itemize}
In particular, 
\begin{eqnarray}
A_{X,\D+I^{(\mu)}_\bu}(E_i) 
= \frac{q_i}{c_i}\cdot A_{X,\D+I^{(\mu)}_\bu}(v_i) 
\le \frac{q_i}{c_i}\cdot C|v_i-v|
< \frac{q_i}{c_i}\cdot C \cdot \frac{\vep}{2 C q_i} 
\le \frac{\vep}{2}. 
\end{eqnarray}

Choose $0< \vep' < \vep/2\ord_{E_i}(I_\bu^{(\mu)})$. Then for $m\gg0$ and general $D_m \in \frac{1}{m}|\CF^{m\mu}R_m|$, we have 
\begin{eqnarray*}  
\lct(X,\D;(1-\vep')D_m) = \lct(X,\D;I_{m,m\mu}^{(1-\vep')/m}) >1, 
\end{eqnarray*}
by \cite[Lemma 1.41]{Xu24} and $\ord_{E_i}(D_m) = \frac{1}{m}\ord_{E_i}(I_{m,m\mu})$ for any $i$. Hence
\begin{eqnarray*}
a_i\,\,\,:=\,\,\,
A_{X,\D+(1-\vep')D_m}(E_i)
%&=& A_{X,\D+I_{m,m\mu}^{(1-\vep')/m}(E_i) - A_{X,\D+(1-\vep')I_\bu^{(\mu)}}(E_i)  \\&& + A_{X,\D+(1-\vep')I_\bu^{(\mu)}}(E_i) - A_{X,\D+I_\bu^{(\mu)}}(E_i) \\&&+ A_{X,\D+I_\bu^{(\mu)}}(E_i) \\
&=& (1-\vep') \Big(\ord_{E_i}(I_\bu^{(\mu)})-\frac{1}{m}\ord_{E_i}(I_{m,m\mu})\Big)\\
&& +\, \vep'\cdot\ord_{E_i}(I_\bu^{(\mu)}) 
   + A_{X,\D+I_\bu^{(\mu)}}(E_i) 
   \,\,\,\le\,\,\, \vep, 
\end{eqnarray*}
since $\ord_{E_i}(\fa_\bu)\le \frac{1}{m}\ord_{E_i}(\fa_m)$ for any graded ideal sequence $\fa_\bu$. 

By \cite[Corollary 1.4.3]{BCHM10}, there exists a $\IQ$-factorial birational model $\pi:Y\to X$ extracts precisely $E_1,\cdots, E_r$. Then 
\begin{eqnarray} 
\label{Eqnarray: crepant pullback 1}
K_Y+\pi_*^{-1}(\D+(1-\vep')D_m)+\sum_{i=1}^r (1-a_i)E_i         = \pi^*(K_X+\D+(1-\vep')D_m). 
\end{eqnarray}
In particular, $\pi^*(K_X+\D+(1-\vep')D_m)\ge K_Y+\pi_*^{-1}\D+(1-\vep)E$. Since $\lct(X,\D;(1-\vep')D_m) >1$, the pair $(Y,\pi_*^{-1}\D+(1-\vep)E)$ is log canonical. 
Hence $(Y,\pi_*^{-1}\D+E)$ is also log canonical by our choice of $\vep$. 

Note that $Y$ is of Fano type by (\ref{Eqnarray: crepant pullback 1}). We may run a $-(K_Y+\pi_*^{-1}\D+E)$-MMP and get a good minimal model $\phi: Y\dashrightarrow \oY$. Let $\opi: \oY\dashrightarrow X$ be the induced birational map and $\oE = \phi_*E$. 
Since $(X,\D+(1-\vep')D_m)$ admits a $\IQ$-complement, we see that $(Y,\pi_*^{-1}\D+(1-\vep)E)$ admits a $\IQ$-complement $\Phi$ by (\ref{Eqnarray: crepant pullback 1}). Let $\oPhi=\phi_*\Phi$. Since 
\begin{eqnarray*}
K_Y+\pi_*^{-1}\D+(1-\vep)E +\Phi
\sim_{\IQ} 0, 
\end{eqnarray*}
the MMP $\phi$ is crepant for the log canonical pair $(Y,\pi_*^{-1}\D+(1-\vep)E+\Phi)$, that is, 
\begin{eqnarray*}
K_Y+\pi_*^{-1}\D+(1-\vep)E +\Phi
= \phi^*(K_\oY+\opi_*^{-1}\D+(1-\vep)\oE+\oPhi). 
\end{eqnarray*}
Hence $(\oY,\opi_*^{-1}\D+(1-\vep)\oE+\oPhi)$ is also log canonical. In particular, $(\oY,\opi_*^{-1}\D+(1-\vep)\oE)$ is log canonical. By our choice of $\vep$, we see that $(\oY,\opi_*^{-1}\D+\oE)$ is also log canonical. 

Recall that $\phi: Y\dashrightarrow \oY$ is a $-(K_Y+\pi_*^{-1}\D+E)$-MMP. We have
\begin{eqnarray}
\label{Eqnarray. (-K)-MMP A_X(v)-decending}
A_{Y,\pi_*^{-1}\D+E}(F)
\ge
A_{\oY,\opi_*^{-1}\D+\oE}(F) \ge 0, 
\end{eqnarray}
for any prime divisor $F$ over $Y$. Since $(\oY,\opi_*^{-1}\D+\oE)$ is log canonical, we see that $\phi$ is an isomorphism at the generic point of each log canonical center $Z$ of $(Y,\pi_*^{-1}\D+E)$ (otherwise, there exists a prime divisor $F$ over $Y$ with $C_Y(F)=Z$ such that 
\begin{eqnarray*}
\label{}
0=A_{Y,\pi_*^{-1}\D+E}(F)
>
A_{\oY,\opi_*^{-1}\D+\oE}(F) \ge 0, 
\end{eqnarray*}
which is a contradiction). In particular, $\phi$ is an isomorphism at any stratum of $E$. Hence
\begin{eqnarray}
\label{Eqnarray. crepant pullback of (-K)-MMP}
\phi^*(K_\oY+\opi_*^{-1}\D+\oE) - (K_Y+\pi_*^{-1}\D+E) = \sum_i (1-A_{\oY,\opi_*^{-1}\D+\oE}(F_i)) F_i \ge 0, 
\end{eqnarray}
by (\ref{Eqnarray. (-K)-MMP A_X(v)-decending}), where $F_i$ are prime divisors extracted by $\phi$. 

Since $Y$ is of Fano type and $\phi: Y\dashrightarrow \oY$ is a birational contraction, we have that $\oY$ is also of Fano type (\cite[Lemma 2.12]{Bir19}). On the other hand, $-(K_{\oY}+\opi_*^{-1}\D+\oE)$ is nef. Hence $-(K_{\oY}+\opi_*^{-1}\D+\oE)$ is semiample and $(\oY,\opi_*^{-1}\D+\oE)$ admits a $\IQ$-complement by Bertini theorem. By (\ref{Eqnarray. crepant pullback of (-K)-MMP}), $(Y,\pi_*^{-1}\D+E)$ also admits a $\IQ$-complement $\Theta$. Then $\Gamma=\pi_*\Theta$ is a $\IQ$-complement of $(X,\D)$ such that $E_i\in \LC(X,\D+\Gamma)$. In particular, $v\in \LC(X,\D+\Gamma)$. 
\end{proof}

\section{Generalized $\BH$-invariants}
\label{Section: Generalized H-invariants}

Let $(X,\D)$ a log Fano pair. In this section, we will define the generalized $\BH$-invariant $\BH^g$ of $(X,\D)$ for any function $g$ satisfying (\ref{Eqnarray: function g}), and study the basic properties of it. 
We fix an Okounkov body $\BO$ of $L=-(K_X+\D)$ with respect to some admissible flag in the following. 

\begin{defi}[$\BH^g$-invariants]\rm 
\label{Definition: H^g invariant}
For any linearly bounded filtration $\CF$ on $R=R(X,\D)$, we define
\begin{eqnarray*}
\BH^g(\CF) \,\,\,=\,\,\, \BH^g_{X,\D}(\CF) 
&:=& 
\log\Big(
\int_\IR g(\mu(\CF) - t) \DH_\CF(\dif t)
\Big) \\
&=& 
\log\Big(
\int_\BO g(\mu(\CF) - G_\CF(y)) \dif y
\Big), \\
h^g(X,\D) 
&:=& 
\inf_\CF \,\,
\BH^g(\CF), 
\end{eqnarray*}
where the infimum runs over all the linearly bounded filtrations $\CF$ on $R$. 
\end{defi}

\begin{rmk}\rm
If we choose $g(x)=e^x$, then $\BH^g$ reveals the original $\BH$-invariant as \cite{TZZZ13,DS20,HL20}, see also \cite[Definition 2.7]{MW24}. It is well-known that $\mu(\CF)$ and $G_\CF$ are affine with respect to shifting, we have $\BH^g(\CF(b))=\BH^g(\CF)$ for any $b\in \IR$. 
\end{rmk}

\subsection{Convexity}
We study the global behavior of $\BH^g$ in the rest of this section. 
Following \cite[Theorem 3.7]{BLXZ23}, we prove the convexity of the $\BH^g$-invariants, which mainly relies on our choice of $g$. As a consequence, we prove the uniqueness of valuative minimizer of $\BH^g$. Let $\CF_0, \CF_1$ be linearly bounded filtrations on $R$. The {\it geodesic} connecting $\CF_0$ and $\CF_1$ is defined by
\begin{eqnarray}
\label{Equality: Geodesic}
\CF^\lam_t R_m 
= \sum_{(1-t)\mu + t\nu\ge \lam}
\CF_0^\mu R_m \cap \CF_1^\nu R_m.
\end{eqnarray}

\begin{thm}
\label{Theorem: gHT: Convexity}
The functional $\BH^g$ is convex along geodesics. More precisely, for any $0\le t\le 1$, we have
$\BH^g(\CF_t)\le (1-t)\BH^g(\CF_0) + t\BH^g(\CF_1). $
%and the equality holds if and only if $d_1(\CF_0, \CF_1(b))=0$ for some $b\in\IR$. 
\end{thm}

\begin{proof}
By \cite[Proposition 3.12]{BLXZ23}, we know that  
\begin{eqnarray}  
\label{Eqnarray. mu convex along geodesic}
\mu(\CF_t) \le (1-t)\mu(\CF_0) + t\mu(\CF_1). 
\end{eqnarray}
Hence 
\begin{eqnarray*}
\BH^g(\CF_t) 
&=& \log\Big(
\int_\IR g(\mu(\CF_t) - s) \DH_{\CF_t}(\dif s)
\Big) \\
&=& \log\Big(
\int_{\IR^2} g(\mu(\CF_t) - (1-t)x-ty) \DH_{\CF_0,\CF_1}(\dif x \dif y)
\Big) \\
&\le& \log\Big(
\int_{\IR^2} g((1-t)(\mu(\CF_0) - x) + t(\mu(\CF_1) - y)) \DH_{\CF_0,\CF_1}(\dif x \dif y)
\Big) \\
&\le& \log\Big(
\int_{\IR^2} g(\mu(\CF_0) - x)^{1-t} \cdot g(\mu(\CF_1) - y)^t \cdot \DH_{\CF_0,\CF_1}(\dif x \dif y)
\Big) \\
&\le& (1-t)\log\Big(
\int_\IR g(\mu(\CF_0) - x) \DH_{\CF_0}(\dif x)
\Big)
+ t\log\Big(
\int_\IR g(\mu(\CF_1) - y) \DH_{\CF_1}(\dif y)
\Big) \\
&=& (1-t)\BH^g(\CF_0) + t\BH^g(\CF_1), 
\end{eqnarray*}
where the first inequality follows from the convexity of $\mu(\CF_t)$ (\ref{Eqnarray. mu convex along geodesic}) and $g$ being increasing, the second one follows from the log concavity of $g$, and the third one follows from H\"older's inequality. 
\end{proof}

%As a consequence, the $\BH^g$-invariant admits at most one minimizer.  \begin{cor} Assume that $\CF_0$, $\CF_1$ are linearly bounded filtrations minimizing $\BH^g$. Then $\CF_0=\CF_1$. \end{cor}

\begin{cor}
\label{Corollary: gHT: uniqueness of minimizer}
Let $v, w$ be valuations over $X$. If $\BH^g(\CF_v)=\BH^g(\CF_w)=h^g(X,\D)$, then $v=w$. 
\end{cor}

\begin{proof}
The proof is slightly different from \cite[Proposition 3.14]{BLXZ23}, which relies on the linearity of $\log \circ g$. Let $\CF_0=\CF_v$ and $\CF_1=\CF_w$, and $\CF_t$ be the geodesic connecting them. Then  
\begin{eqnarray*}  
\BH^g(\CF_t)\le (1-t)\BH^g(\CF_0) + t\BH^g(\CF_1) = h^g(X,\D). 
\end{eqnarray*}
So the equality holds, hence do those in the proof of Theorem \ref{Theorem: gHT: Convexity}. Then since we used H\"older's inequality, we have $g(\mu(\CF_0)-x)=c\cdot g(\mu(\CF_1)-y)$ almost everywhere on $\IR^2$ with respect to the measure $\DH_{\CF_0,\CF_1}$ for some $c>0$. On the other hand, since $\BH^g(\CF_0)=\BH^g(\CF_1)$, we have $c=1$. Hence $\mu(\CF_0)-x=\mu(\CF_1)-y$ almost everywhere on $\IR^2$ with respect to the measure $\DH_{\CF_0,\CF_1}$ since $g$ is continuous and strictly increasing, that is, 
\begin{eqnarray*}  
0=\int_{\IR^2}|x-y-d|\DH_{\CF_0,\CF_1}(\dif x \dif y) = d_1(\CF_0, \CF_1(d)), 
\end{eqnarray*}
where $d = \mu(\CF_0)-\mu(\CF_1)$. Then $\CF_0$ and $\CF_1(d)$ are equivalent, so they have the same $\lam_\min$, and $d=0$ by \cite[Lemma 2.5]{BLXZ23}. We conclude that $v=w$ by \cite[Proposition 2.27]{HL20} or \cite[Lemma 3.16]{BLXZ23}. 
\end{proof}

%Another corollary is the behavior of $\BH^g$ on a quasi-monomial simplicial cone $\sigma=\QM_\eta(Y,E)$, where $(Y,E)\to (X,\D)$ is a log smooth model and $\eta$ is the generic point of some stratum of $E$. In this case, the geodesic connecting $v,w\in \sigma$ is the obvious line segment in $\sigma$. 

\begin{thm}
\label{Theorem: beta^g on simplicial cone}

Let $\Gamma$ be a $\IQ$-complement of $(X,\D)$ and $\sigma\seq \LC(X,\D+\Gamma)$ be a quasi-monomial simplicial cone. Then the function 
$$\BH^g|_\sigma : \sigma \to \IR,\quad  v\mapsto \BH^g(v) $$
is continuous and admits a minimizer $v_0\in \sigma$. 
\end{thm}
\begin{proof}
%With the same argument as Corollary \ref{Corollary: gHT: uniqueness of minimizer}, The function $\BH^g: \sigma\to \IR_{>0}$ is strictly convex and admits at most one minimizer. To see the existence, it suffice to show that for any $v\in \sigma\setminus\{0\}$,  $\BH^g(a\CF_v) \to +\infty$ as $a\to +\infty$, which holds since $g$ is strictly increasing. 

For the continuity of $\BH^g|_\sigma$, we follow the argument of \cite[Lemma 4.15]{BLXZ23}. 
For any $v\in \sigma \seq \LC(X,\D+\Gamma)$, we have $\mu(\CF_v) = A_{X,\D}(v)$ by Lemma \ref{Lemma: mu=A}. Hence the function $v\mapsto \mu(\CF_v) = A_{X,\D}(v)$ is linear on $\sigma$ by \cite[Proposition 5.1]{JM12}.

For any $v\in \sigma$ and $x\in \IR_{\ge 0}$, we denote by 
$$f_v(x) = e^{\BH^g(xv)} = 
\int_\BO g(x(A_{X,\D}(v) - G_v(y))) \dif y. $$
Then it is smooth since $g$ is smooth, and 
\begin{eqnarray*}  
f_v'(x) &=& 
\int_\BO (A_{X,\D}(v) - G_v(y)) g'(x(A_{X,\D}(v) - G_v(y))) \dif y, \\
f_v''(x) &=& 
\int_\BO (A_{X,\D}(v) - G_v(y))^2 g''(x(A_{X,\D}(v) - G_v(y))) \dif y > 0, 
\end{eqnarray*}
where $f_v'' >0$ since $\log \circ g$ is convex, that is, $g'' \ge (g')^2/g > 0$. Then $f'$ admits at most one zero-point. Choose $x_0>0$ such that $f'(x_0)>0$. Hence $f_v$ is strictly convex and increasing for $x\ge x_0$. We see that $\lim_{x\to +\infty} f_v(x) = +\infty$.

Choose a log smooth model $(Y,E=E_1+\cdots + E_r) \to (X,\D)$ such that $\sigma = \QM_\eta(Y,E)$, where $\eta$ is the generic point of some component of $E_1\cap\cdots \cap E_r$. Let $z_1,\cdots , z_r$ be the regular functions in $\CO_{Y,\eta}$ satisfying $E_i = \{z_i=0\}$. For any $\alpha = (\alpha_1,\cdots, \alpha_r) \in \IR^r_{\ge 0}$, we denote by $v_\alpha \in \sigma$ the corresponding quasi-monomial valuation satisfying $v_\alpha(z_i) = \alpha_i$. 

For any $\alpha \in \IR^r_{\ge 0}$ and any sequence $\IR_{\ge 0}\ni \alpha^{(k)} \to \alpha$ (as $k\to\infty$), we denote by $v^{(k)} = v_{\alpha^{(k)}}$ and $v= v_\alpha$. We will show that $\lim_{k\to \infty} \BH^g(v^{(k)}) = \BH^g(v)$. 
Note that $\BH^g(xv)$ is continuous in $x\in \IR_{\ge 0}$. Therefore, after rescaling the $\alpha^{(k)}$ and removing finitely many terms, we may assume the sequence $\alpha^{(k)}$ is non-increasing. Hence $v^{(k)} \ge v^{(k+1)} \ge v$. Let $\BO\seq \IR_{\ge 0}^n$ be an Okounkov body of $(X,\D)$ and $G^{(k)}, G: \BO\to \IR_{\ge 0}$ be the concave transform of $v^{(k)}$ and $v$ respectively. Then we have $G^{(k)} \le G^{(k+1)} \le G$. By \cite[Proposition 2.4]{BLX19}, we have $S(v^{(k)}) \to S(v)$ as $k\to \infty$, that is, 
\begin{eqnarray*}  
0\le \int_\BO (G(y)-G^{(k)}(y)) \dif y \to 0, \quad k\to \infty. 
\end{eqnarray*}
Hence $G^{(k)} \to G$ almost everywhere as $k\to \infty$. So $g(A_{X,\D}(v^{(k)}) - G^{(k)}) \to g(A_{X,\D}(v) - G)$ almost everywhere as $k\to \infty$. We conclude that $\BH^g(v^{(k)}) \to \BH^g(v)$ as $k\to \infty$. 
\end{proof}

\subsection{Approximation by valuations}

\begin{defi}[$\tbeta^g$-invariants]\rm
For any valuation $v$ over $X$ with $A_{X,\D}(v)<+\infty$, we define 
\begin{eqnarray*}  
\tbeta^g(v) 
\,\,\,=\,\,\, 
\tbeta^g_{X,\D}(v) 
\,\,\,:=\,\,\, 
\log\Big(
\int_\IR g(A_{X,\D}(v) - t) \DH_{\CF_v}(\dif t)
\Big). 
\end{eqnarray*}
By convention, we set $\tbeta^g(v)=+\infty$ if $A_{X,\D}(v)=+\infty$. 
\end{defi}

\begin{rmk} \rm 
\label{Remark: H^g = beta^g}
Since $\mu_{X,\D}(\CF_v)\le A_{X,\D}(v)$, we have naturally $\BH^g(\CF_v) \le \tbeta^g(v)$. The equality holds if $v$ is a log canonical place of $(X,\D+\Gamma)$ by Lemma \ref{Lemma: mu=A}, where $\Gamma\in |L|_\IQ$ such that $(X,\D+\Gamma)$ is log canonical. 
\end{rmk}

We have shown that the $\BH^g$-invariants admit at most one valuative minimizer. For the existence, we prove the following theorem as preparation. 

\begin{thm}
\label{Theorem: gHT: Valuative criterion}
$h^g(X,\D) = \inf_{v\in \Val_X} \, \tbeta^g(v). $
\end{thm}

\begin{proof}
We need to show that for any linearly bounded filtration $\CF$ on $R_\bu$, there exists a valuation $v$ over $X$ such that $\BH^g(\CF) \ge \tbeta^g(v)$. 

Just assume that $\mu = \mu(\CF) < \lam_\max(\CF)$. Then we have $\lct(X,\D;I^{(\mu)}_\bu) = 1$. There exists a valuation $v$ on $X$ computing $\lct(X,\D;I^{(\mu)}_\bu)$ by \cite{JM12}. Hence $v(I^{(\mu)}_\bu)\le A_{X,\D}(v)$. 
We denote by $f_v(t)=v(I^{(t)}_\bu)$, which is a convex function on $\IR$. Rescale $v$ such that the first order left-derivative at $\mu\in\IR$ equals to one, that is, $f'_{v,-}(\mu)=1$. Then we have 
\begin{eqnarray}
\label{Inequality: convexity of f_v(t)}
f_v(t)
\ge t + f_v(\mu) - \mu 
\ge t + A_{X,\D}(v)-\mu. 
\end{eqnarray}
We claim that $\CF':=\CF(A_{X,\D}(v)-\mu) \seq \CF_v$ (that is, $\CF'^\lam R_m \seq \CF^\lam R_m$ for any $m\in l_0\IN$ and $\lam \in\IR$), hence $G_{\CF'}\le G_{\CF_v}$. Indeed, for any $\lam \in \IR$ and $s \in \CF^{m(\lam - A_{X,\D}(v)+\mu)}R_m$, 
\begin{eqnarray*}
\frac{1}{m}v(s) 
\ge \frac{1}{m}
v(I_{m,m(\lam - A_{X,\D}(v)+\mu)})
\ge f_v(\lam - A_{X,\D}(v)+\mu)
\ge \lam, 
\end{eqnarray*}
where the third inequality follows from (\ref{Inequality: convexity of f_v(t)}) with $t=
\lam - A_{X,\D}(v)+\mu$. Hence $s\in \CF^{m\lam}_v R_m$. Recall that the functional $\mu(\CF)$ and measure $\DH_{\CF}$ are affine with respect to shift of filtrations, that is, $\mu(\CF(b))=\mu(\CF)+b$ and $\int_\IR f(s) \DH_{\CF(b)}(\dif s) = \int_\IR f(s+b) \DH_{\CF}(\dif s)$ for any $b\in \IR$. Hence $\BH^g(\CF) = \BH^g(\CF(b))$. 
We conclude that 
\begin{eqnarray*}
\BH^g(\CF) 
\,\,\,=\,\,\, \BH^g(\CF') 
&=& \log\Big(
\int_\BO g(\mu(\CF') - G_{\CF'}(y)) \dif y
\Big) \\
&=&  \log\Big(
\int_\BO g(A_{X,\D}(v) - G_{\CF'}(y)) \dif y
\Big) \\
&\ge& \log\Big(
\int_\BO g(A_{X,\D}(v) - G_{\CF_v}(y)) \dif y
\Big) 
\,\,\,=\,\,\,\tbeta^g(v). 
\end{eqnarray*}
The proof is finished. 
\end{proof}

\begin{rmk} \rm
In the theorem $v\in \Val_X$ can be replaced by $v$ being quasi-monomial valuations over $X$. Indeed, in the proof we can choose a quasi-monomial minimizer of $\lct(X,\D;I^{(\mu)}_\bu)$ by \cite{Xu19}. 
\end{rmk}

\subsection{Weighted delta invariants}
By \cite[Definition 4.1]{BLXZ23}, we define the following version of weighted delta invariants. This is one of the key ingredients in the proof of speciality of $\BH^g$-minimizer in the next section. 
%In the following, we assume that there exists a valuation $v_0$ minimizing $\BH^g$ and $\mu(\CF_{v_0})=A_{X,\D}(v_0)$. 

Let $g':\IR\to \IR_{>0}$ be the first order derivative of $g$. Recall that $R_m=H^0(X,mL)$ and $N_m=\dim R_m$. 

\begin{defi}\rm
%We follow the arguments in \cite[Section 5]{BLXZ23}. Roughly speaking, we will show that $v_0$ is the minimizer of some ``$\delta$-invariant'', hence is special by a similar argument of \cite[Lemma 3.1]{LXZ22}. 
Let $\CF_0, \CF$ be linearly bounded filtrations on $R$, and $\mu_0=\mu(\CF_0)$, we define 
\begin{eqnarray*} 
N^{g',\CF_0}_m 
&:=& \sum_{i=1}^{N_m} g'\Big(\mu_0-\frac{\ord_{\CF_0}(s_i)}{m}\Big),\\
S^{g',\CF_0}_m(\CF)
\,\,\,=\,\,\,
S^{g',\CF_0}_m(L;\CF) 
&:=& \frac{1}{N^{g',\CF_0}_m} 
\sum_{i=1}^{N_m} g' \Big(\mu_0-\frac{\ord_{\CF_0}(s_i)}{m}\Big)\cdot 
\frac{\ord_\CF(s_i)}{m}, 
\end{eqnarray*} 
where $\{s_i\}$ is a basis of $R_m$ which is compatible with both $\CF_0$ and $\CF$ (such a basis exists by \cite[Lemma 3.1]{AZ22}). It is clear that $S^{g',\CF_0}_m(L;\CF)$ does not depend on the choice of $\{s_i\}$. Let
\begin{eqnarray*} 
S^{g',\CF_0}(\CF) 
\,\,\,=\,\,\,
S^{g',\CF_0}(L;\CF) 
\,\,\,:=\,\,\,
\mathop{\lim}_{m\to \infty} S^{g',\CF_0}_m(L;\CF) 
\,\,\,=\,\,\, \frac{\int_{\IR^2} g'(\mu_0-x)y \cdot \DH_{\CF_0,\CF}(\dif x \dif y)}{\int_{\IR} g'(\mu_0-x) \cdot \DH_{\CF_0}(\dif x)}, 
\end{eqnarray*} 
where the last equality follows from the weakly convergence of $\DH_{\CF_0,\CF}$ \cite[Section 3.1.3]{BLXZ23}. Finally let 
\begin{eqnarray*} 
\delta^{g',\CF_0}_m(X,\D;L)
\,\,\,:=\,\,\, \inf_v \frac{A_{X,\D}(v)}{S^{g',\CF_0}_m(L;v)}, \qquad
\delta^{g',\CF_0}(X,\D;L)
\,\,\,:=\,\,\, \inf_v \frac{A_{X,\D}(v)}{S^{g',\CF_0}(L;v)}, 
\end{eqnarray*} 
where the infimum runs over all the valuations $v$ over $X$. 
\end{defi}

We have the following generalization of \cite[Theorem 5.1]{BLXZ23}. 

\begin{lem}
\label{Lemma: gHT: H-minimizer v_0 and delta^(g,v_0)}
Let $\CF_0$ be a linearly bounded filtration on $R=R(X;L)$ with $\mu_0=\mu(\CF_0)$ and $v_0$ be a valuation minimizing $\lct(X,\D;I_\bu^{(\mu_0)})$. By shifting $\CF_0$, we may assume that $\mu_0 = A_{X,\D}(v_0)$. 

Then $\CF_0$ minimizes $\BH^g$ if and only if 
$\delta^{g',\CF_0}(X,\D;L)=\frac{A_{X,\D}(v_0)}{S^{g',\CF_0}(L;v_0)}=1$ and $\BH^g(\CF_0)=\tbeta^g(v_0)$. 
\end{lem}

\begin{proof}
The proof follows from \cite[Theorem 5.1]{BLXZ23}. We first prove the ``if'' part. By Theorem \ref{Theorem: gHT: Valuative criterion}, it suffices to show $\tbeta^g(v)\ge \BH^g(\CF_0)$ for any valuation $v$ over $X$. 

Let $\BO=\BO(R_\bu)$ be an Okounkov body of the anti-canonical ring $R_\bu$. By the proof of Theorem \ref{Theorem: gHT: Valuative criterion}, we know that $\CF_0\seq \CF_{v_0}$, hence $G_{\CF_0}\le G_{\CF_{v_0}}$. The assumptions $\mu_0 = A_{X,\D}(v_0)$ and $\BH^g(\CF_0)=\tbeta^g(v_0)$ imply that $G_{\CF_0}= G_{\CF_{v_0}}$ almost everywhere on $\BO$. Hence 
\begin{eqnarray}
\label{Eqnarray: S(CF_0)=S(v_0)}
S^{g',\CF_0}(\CF_0)=S^{g',\CF_0}(v_0).  
\end{eqnarray}
Let $\CF_t$ be the geodesic connecting $\CF_0$ and $\CF_1:= \CF_{v}$. We define the following analog of $\BH^g(\CF_t)$,
\begin{eqnarray*} 
f(t) 
&:=& \log\Big(
\int_{\IR^2} g((1-t)(\mu_0-x)+t(A_{X,\D}(v)-y) ) \DH_{\CF_0,\CF_1}(\dif x \dif y)
\Big). 
\end{eqnarray*} 
Then similar argument of Theorem \ref{Theorem: gHT: Convexity} shows that $f$ is convex. 
We have 
\begin{eqnarray*} 
f'(0) 
&=& e^{-f(0)}\cdot 
\int_{\IR^2} \big((A_{X,\D}(v)-y)-(\mu_0-x)\big) g'(\mu_0-x) \DH_{\CF_0,\CF_1}(\dif x \dif y)
, \\
&=& e^{-f(0)} \Bv^{g',\CF_0}\cdot 
\Big( (A_{X,\D}(v)   - S^{g',\CF_0}(v))   
    -(\mu_0 - S^{g',\CF_0}(\CF_0))
\Big), \\
&=& e^{-f(0)} \Bv^{g',\CF_0}\cdot 
   (A_{X,\D}(v)   - S^{g',\CF_0}(v)) 
\,\,\,\ge\,\,\, 0,
\end{eqnarray*} 
where $\Bv^{g',\CF_0} = \int_{\IR} g'(\mu_0-x) \DH_{\CF_0}(\dif x)$ and the third equality follows from (\ref{Eqnarray: S(CF_0)=S(v_0)}). Hence 
\begin{eqnarray*}  
\BH^g(\CF_0)=f(0)\le f(1)=\tbeta^g(v). 
\end{eqnarray*}

Next, we prove the ``only if'' part. 
By Theorem \ref{Theorem: gHT: Valuative criterion}, we know that $\BH^g(\CF_0)\ge \tbeta^g(v_0)\ge \BH^g(\CF_{v_0})$. Hence both the equalities hold since $\CF_0$ minimizes $\BH^g$, and we also have (\ref{Eqnarray: S(CF_0)=S(v_0)}). 

For any valuation $v$ over $X$, let $\CF_t$ and $f$ be the same as above. 
Since $\mu(\CF_v) \le A_{X,\D}(v)$, we have 
\begin{eqnarray*} 
\mu(\CF_t) 
\le (1-t)\mu(\CF_0) + t\mu(\CF_1) 
\le (1-t)\mu_0 + tA_{X,\D}(v). 
\end{eqnarray*} 
Hence $f(0)=\BH^g(\CF_0)\le\BH^g(\CF_t) \le f(t)$ for any $0\le t\le 1$. We conclude that $f'(0)\ge 0$ since $f$ is convex, that is, 
\begin{eqnarray*}  
A_{X,\D}(v) - S^{g',\CF_0}(v)
\ge \mu_0 - S^{g',\CF_0}(\CF_0) = A_{X,\D}(v_0) - S^{g',\CF_0}(v_0), 
\end{eqnarray*}
by the assumption and (\ref{Eqnarray: S(CF_0)=S(v_0)}). 
If $v=\lam v_0$, we see that
$$(\lam-1)(A_{X,\D}(v_0) - S^{g',\CF_0}(v_0))\ge 0, $$
for any $\lam > 0$. Hence $A_{X,\D}(v_0) - S^{g',\CF_0}(v_0) = 0$. The proof of Lemma \ref{Lemma: gHT: H-minimizer v_0 and delta^(g,v_0)} is finished. 
\end{proof}

%Firstly, $v_0$ is quasi-monomial. Indeed, we have $\delta^{g,v_0}_m \to \delta^{g, v_0}$ as $m\to \infty$, and $\delta^{g,v_0}_m$ is minimized by some weakly special divisor $E_m$. 
%it is the limit of a sequence of weakly special divisorial valuations, and then using Birkar's boundedness of complements. 

\section{Existence of $\BH^g$-minimizers and finite generation}
\label{Section: Existence of H^g-minimizers and finite generation}

In this section, let $(X,\D)$ be a log Fano pair and $L=-(K_X+\D)$.

\subsection{Approximation by test configurations}

Recall that a {\it normal test configuration (TC)} of $(X,\D)$ is a collection $(\CX, \D_\CX;\CL,\eta)$ consisting of
\begin{itemize}
\item A normal variety $\CX$ with a $\IG_m$-action generated by $\eta\in \Hom(\IG_m, \Aut(\CX))$; 
\item A $\IG_m$-equivariant morphism $\pi: \CX\to \IA^1$, where the $\IG_m$-action on $\IA^1$ is standard; 
\item A $\IG_m$-equivariant $\pi$-semiample $\IQ$-Cartier divisor $\CL$ on $\CX$; 
\item A $\IG_m$-equivariant trivialization over the punctured plane $i_\eta:(\CX,\CL)|_{\pi^{-1}(\IG_m)}\cong (X,L)\times \IG_m$, which is compatible with $\pi$ and $\pr_1$. And $\D_\CX$ is the closure of $i_\eta^{-1}(\D\times\IG_m)$ in $\CX$.  
\end{itemize}

The test configuration $(\CX,\D_\CX;\CL,\eta)$ is called {\it (weakly) special} if $(\CX,\CX_0+\D_\CX)$ is (lc) plt, and $\CL=-K_{\CX/\IA^1}-\D_\CX + c\CX_0$ for some $c\in\IQ$. Note by adjunction that $(\CX, \D_\CX;\CL,\eta)$ being special is equivalent that the central fiber $(\CX_0, \D_{\CX,0})$ is a log Fano pair. 

For any test configuration $(\CX,\D_\CX;\CL,\eta)$ of $(X,\D)$, we have the following
$\IZ$-filtration $\CF=\CF_{(X,\D_\CX;\CL,\eta)}$ on the anti-canonical ring $R=R(X,\D)$, 
\begin{eqnarray} \label{Formula: test configuration to filtration}
\CF^\lam R_m
&:=&\{f\in H^0(X,mL)\mid t^{-\lam} \bar{f} \in H^0(\CX,m\CL)\}, 
\end{eqnarray}
where $t$ is the parameter on $\IA^1$, and $\bar{f}$ is the $\IG_m$-extension of $f$ on $\CX\setminus\CX_0$ and viewed as a rational section of $m\CL$. 
We simply denote $\BF(\CF_{(X,\D_\CX;\CL,\eta)})$ by $\BF(X,\D_\CX;\CL,\eta)$ for $\BF=\BL$ or $\BH^g$. 
%By \cite[Example 2.14]{HL20}, we have $$a\CF(b)=\CF_{(X,\D_\CX;\CL+b\CX_0,a\eta)}, $$ for any $a\in \IR_{>0}$ and $b\in \IR$. 
We have 
\begin{eqnarray} \label{Formula: test configuration to filtration}
\BL(\CX,\D_\CX;\CL,\eta):=\lct(\CX,\D_\CX+\CD;\CX_0)-1, 
\end{eqnarray}
where $\CD\sim_\IQ -(K_\CX+\D_\CX)-\CL$ is supported on $\CX_0$, see for example \cite[Theorem 3.66]{Xu24}. 

%Indeed, since the Rees algebra of $\CF$ induced by test configuration is finitely generated, there exists $m_0\in l_0\IN$ such that $\CI_{m_0l}= \CI_{m_0}^l$ for any $l\in\IN$. Hence $$\BL(X,\D_\CX;\CL,\eta):=\lct(X_{\IA^1},\D_{\IA^1}+\CI_{m_0}^\frac{1}{m_0};X_0)-1, $$

Conversely, for any linearly bounded filtration $\CF$ on $R$, one may construct a sequence of test configuration $(\CX_m;\CL_m)$ approximating it, see for example \cite[Definition 3.65]{Xu24}. We shortly recall the construction. Recall that $\CI_m(e_+) \seq \CO_X[s]$ is the graded ideal sequence associated to $\CF$ in Definition \ref{Definition: L-functional}.  
Let $\pi_m: \CX_m\to X_{\IA^1}$ be the normalized blowup along $\CI_m(e_+)$ with exceptional divisor $\CE_m$, and $\D_{\CX_m} = \pi_{m,*}^{-1} \D_{\IA^1}$. Then $\CL_m = \pi_m^* L_{\IA^1} - \frac{1}{m} \CE_m$ is semiample by \cite[Lemma 3.64]{Xu24}. Hence $(\CX_m, \D_{\CX_m};\CL_m,\eta_m)$ is a normal test configuration of $(X,\D)$ and is called the {\it $m$-th approximating TC} of $\CF$. We remark that the definition depends on the choice of $e_+$. 
\begin{lem}\cite[Proposition 2.16 and 2.28]{HL20}
\begin{eqnarray} 
\label{Eqnarray: L-estimate}
\BL(\CF) &\ge& \mathop{\lim}_{m\to \infty} \BL(\CX_m, \D_{\CX_m};\CL_m,\eta_m), \\
\label{Eqnarray: DH-convergence}
\DH_\CF &=& \mathop{\lim}_{m\to \infty} \DH_{(\CX_m, \D_{\CX_m};\CL_m,\eta_m)}. 
\end{eqnarray}
\end{lem}

%We remark that (\ref{Eqnarray: L-estimate}) only holds for Fano varieties, but (\ref{Eqnarray: DH-convergence}) holds for polarized varieties. 

\begin{cor}
\begin{eqnarray} 
\BH^g(\CF) \ge \mathop{\lim}_{m\to \infty} \BH^g(\CX_m, \D_{\CX_m};\CL_m,\eta_m), 
\end{eqnarray}
\end{cor}

\begin{thm} 
For any log Fano pair $(X,\D)$, we have 
\begin{eqnarray} 
h^g(X,\D) = \mathop{\inf}_{(\CX,\D_\CX;\CL,\eta)} \BH^g(\CX, \D_{\CX};\CL,\eta), 
\end{eqnarray}
where the infimum runs over all the normal test configurations $(\CX, \D_{\CX};\CL,\eta)$ of $(X,\D)$. 
\end{thm}

For any test configuration $(\CX,\D_\CX;\CL,\eta)$ of $(X,\D)$, we denote by $-\CD=\CL+(K_\CX+\D_\CX) = \sum_i e_i E_i$ and $\CX_0 = \sum_i b_i E_i$, where $E_i\seq \CX$ are irreducible components of $\CX_0$. Let $v_i=\ord_{E_i}|_{\CX_1}$ be the corresponding divisorial valuations over $X=\CX_1$. We have the following description of the filtration $\CF=\CF_{(\CX,\D_\CX;\CL,\eta)}$ induced by $(\CX,\D_\CX;\CL,\eta)$. 

\begin{lem}
\label{Lemma: BHJ17}
The filtration $\CF=\CF_{(\CX,\D_\CX;\CL,\eta)}$ induced by $(\CX,\D_\CX;\CL,\eta)$ is
\begin{eqnarray*} 
\CF_{(\CX,\D_\CX;\CL,\eta)}
\,\,\,=\,\,\, \bigcap_i b_i^{-1} \Big( \CF_{v_i}(e_i+1-b_i - A_{X,\D}(v_i)) \Big). 
\end{eqnarray*}
In other words, 
\begin{eqnarray*} 
\CF_{(\CX,\D_\CX;\CL,\eta)}^\lam R_m
\,\,\,=\,\,\, \bigcap_i \{s\in R_m \mid v_i(s) + e_i+1-b_i - A_{X,\D}(v_i) \ge b_i\lam  \},  
\end{eqnarray*}
\end{lem}
\begin{proof}
Let $\CY$ be the graph of the birational map $\CX\dashrightarrow X_{\IA^1}$, and $\pi:\CY\to \CX$, $\tau:\CY\to X_{\IA^1}$ be the corresponding morphisms. 
\[
\xymatrix{
  & \CY \ar^-{\tau}[dr]\ar_-{\pi}[dl]&  \\
  \CX &      & X_{\IA^1}. } 
\]
By \cite[Lemma 5.17]{BHJ17} (whose notation is $v_{E_i} = b_i^{-1}v_i$), for any $\lam$ and $m$, we have 
\begin{eqnarray*} 
\CF_{(\CX,\D_\CX;\CL,\eta)}^\lam R_m 
\,\,\,=\,\,\, \bigcap_i \CF_{v_i}^{b_i\lam - m\cdot\ord_{E_i}(D)} R_m,  
\end{eqnarray*}
where $D = \pi^*\CL - \tau^*L_{\IA^1}$ is supported on $\CY_0$. It suffices to prove $\ord_{E_i}(D) = e_i+1-b_i - A_{X,\D}(v_i)$. Since 
\begin{eqnarray*} 
D = \pi^*(\CL + K_\CX +\D_\CX) + ( -\pi^*( K_\CX +\D_\CX) - \tau^*L_{\IA^1}) 
= \sum_{i} e_i E_i + B, 
\end{eqnarray*}
where $B = -\pi^*( K_\CX +\D_\CX) + \tau^*(K_{X_{\IA^1}}+ \D_{\IA^1})$ is supported on $\CY_0$. 
By Lemma \ref{Lemma: birational triangle}, we have 
\begin{eqnarray*} 
\ord_{E_i}(B) = A_{\CX,\D_\CX}(E_i) - A_{X_{\IA^1}, \D_{\IA^1}}(E_i) = 1-(b_i+A_{X_{\IA^1}, \D_{\IA^1}+X_0}(E_i)) = 1-b_i-A_{X,\D}(v_i), 
\end{eqnarray*}
where the second and third equalities follows from $\ord_{E_i}(X_0)=b_i$ and adjunction respectively, see for example \cite[Proposition 4.1]{BHJ17}. 
\end{proof}

\begin{lem}
\label{Lemma: birational triangle}
Let $\pi:Z\to (X,\D_X)$ and $\tau:Z\to (Y,\D_Y)$ be birational morphisms of $\IQ$-Gorenstein families over a curve $C$, which are isomorphisms away from $0\in C$, and $\Supp(\D_X), \Supp(\D_X)$ do not contain any fiber of the families. 
Then for any irreducible component $E$ of $Z_0\seq Z$, we have
$$\ord_E(-\pi^*(K_X+\D_X)+\tau^*(K_Y+\D_Y))=A_{X,\D_X}(E)-A_{Y,\D_Y}(E). $$
\end{lem}
%\[\xymatrix{  & Z \ar^-{\tau}[dr]\ar_-{\pi}[dl]&  \\
%X &      & Y. } \]
\begin{proof}
Note that 
\begin{eqnarray*} 
\pi^*(K_X+\D_X)=K_Z+\pi^{-1}_*\D_X+(1-A_{X,\D_X}(E))E + F, \\
\tau^*(K_Y+\D_Y)=K_Z+\tau^{-1}_*\D_Y+(1-A_{Y,\D_Y}(E))E + F',
\end{eqnarray*}
where $F, F' \seq Z_0$ are $\IQ$-divisors that do not contain $E$ as a component. By assumption, we have $\pi^{-1}_*\D_X=\tau^{-1}_*\D_Y$. Hence
\begin{eqnarray*} 
B=-\pi^*(K_X+\D_X)+\tau^*(K_Y+\D_Y)= (A_{X,\D_X}(E) - A_{Y,\D_Y}(E))E+ F'-F,  
\end{eqnarray*}
is a $\IQ$-divisor supported in $Z_0$. We conclude that $\ord_E(B)=A_{X,\D_X}(E) - A_{Y,\D_Y}(E)$. 
\end{proof}

\subsection{Approximation by special test configurations}

The following theorem is an analog of \cite[Theorem 3.4]{HL20}, which depends on Li-Xu's proof of Tian's conjecture \cite{LX14}. Different from Han-Li's proof which relies on an analytic description of the $\BH$-invariants, we give a pure algebraic proof by considering the filtrations induced by test configurations. 

\begin{thm}
For any normal test configuration $(\CX,\D_\CX;\CL,\eta)$ of $(X,\D)$ and $a\in\IR_{>0}$, there exists a special test configuration $(\CX^s,\D_{\CX^s};\CL^s,\eta^s)$ and $a^s\in\IR_{>0}$ such that 
\begin{eqnarray*}  
\BH^g(\CX^s,\D_{\CX^s};\CL^s,a^s\eta^s) \le
\BH^g(\CX,\D_\CX;\CL,a\eta). 
\end{eqnarray*}
\end{thm}  

\begin{proof}
We follow the proof of \cite[Theorem 3.4]{HL20}. 

{\bf Step 1.} (Semistable reduction $\CX^{(d_1)}$). By \cite[Lemma 5]{LX14}, there exists a semistable reduction $\CX^{(d_1)} \to \CX$ over $\IA^1\to \IA^1, z\mapsto z^{d_1}$, such that $\CX^{(d_1)}_0$ is reduced. Since the filtration
\begin{eqnarray*}
\CF_{(\CX^{(d_1)}, \D_{\CX^{(d_1)}}; \CL^{(d_1)}, \frac{a}{d_1} \eta^{(d_1)} )} 
= \CF_{(\CX,\D_\CX;\CL,a\eta)}
\end{eqnarray*}
is not changed, the $\BH^g$-invariants are the same. 

{\bf Step 2.} (Lc modification $\CX^\lc$). By \cite[Theorem 2]{LX14}, which is proved by running a $\IG_m$-equivariant MMP on a log resolution of $(\CX^{(d_1)}, \D_{\CX^{(d_1)}}+\CX^{(d_1)}_0)$, there is a $\IG_m$-equivariant log canonical modification $\pi^\lc:\CX^\lc \to \CX^{(d_1)}$ such that $(\CX^\lc,\D_{\CX^\lc}+\CX^\lc_0)$ is log canonical and $K_{\CX^\lc}+\D_{\CX^\lc}$ is ample over $\CX^{(d_1)}$. 

Write $E=\CL^{(d_1)} + K_{\CX^\lc} + \D_{\CX^\lc}  = \sum_{i=1}^l e_i E_i$ with $e_1\le e_2\le\cdots\le e_l$, where $E_i$ are irreducible components of $\CX^\lc_0$. Let $\CL^\lc_\lam = \CL^{(d_1)} + \lam E= -(K_{\CX^\lc} + \D_{\CX^\lc})+(1+\lam)E$ and 
$\CF_\lam := \CF_{(\CX^\lc, \D_{\CX^\lc}; \CL^\lc_\lam, \eta^\lc)}$. By Lemma \ref{Lemma: BHJ17}, we have
\begin{eqnarray*}
\frac{a}{d_1}\CF_\lam
&=& \CF_{(\CX^\lc, \D_{\CX^\lc}; \CL^\lc_\lam, \frac{a}{d_1} \eta^\lc)}
\,\,\,=\,\,\, 
\frac{a}{d_1} \bigcap_i \Big(\CF_{v_i}((1+\lam)e_i-A_{X,\D}(v_i)) \Big), \\
G_{\CF_\lam}(y) 
&=&
\min_i \, \Big( G_{v_i}(y)+(1+\lam)e_i-A_{X,\D}(v_i) \Big),
\,\,\, \forall y\in \BO. 
\end{eqnarray*} 
On the other hand, by \cite[Example 2.31]{HL20} we have 
\begin{eqnarray*}
\BL(\CF_\lam) 
&=&
\BL(\CX^\lc, \D_{\CX^\lc}; \CL^\lc_\lam, \eta^\lc)
\,\,\,=\,\,\,  (1+\lam) e_1. 
\end{eqnarray*} 
If $\lam=0$, we have
\begin{eqnarray*}
\frac{a}{d_1}\CF_0
\,\,\,=\,\,\, 
\CF_{(\CX^\lc, \D_{\CX^\lc}; \CL^\lc_0, \frac{a}{d_1} \eta^\lc)}
\,\,\,=\,\,\,
\CF_{(\CX^{(d_1)}, \D_{\CX^{(d_1)}}; \CL^{(d_1)}, \frac{a}{d_1} \eta^{(d_1)} )}.  
\end{eqnarray*}
We denote by $i(y)$ the minimizer of the above minimum for any $y\in \BO$. Then 
\begin{eqnarray*}
\BH^g\Big(\frac{a}{d_1}\CF_\lam\Big)
&=& \log\Big( \int_\BO g\Big(\frac{a}{d_1} 
\big( \BL(\CF_\lam)-G_{\CF_\lam}(y) \big)
\Big) \dif y \Big) \\
&=& \log\Big( \int_\BO g\Big(\frac{a}{d_1}
\max_i\big( (1+\lam)(e_1-e_i) + A_{X,\D}(v_i) -G_{v_i}(y) \big)
\Big) \dif y \Big), \\
\frac{\dif }{\dif \lam}
\BH^g\Big(\frac{a}{d_1}\CF_\lam\Big)
&=&  \frac{a}{d_1}
\frac{\int_\BO (e_1-e_{i(y)}) \cdot 
g'\circ f(\lam,y) \dif y}{\int_\BO g\circ f(\lam,y) \dif y}  \le 0, 
\end{eqnarray*}
where
$f(\lam,y) = \frac{a}{d_1} 
\big( \BL(\CF_\lam)-G_{\CF_\lam}(y) \big)$. Recall that $K_{\CX^\lc}+\D_{\CX^\lc}$ is ample over $\CX^{(d_1)}$, so is $E=\CL^{(d_1)}+K_{\CX^\lc} + \D_{\CX^\lc}$. Hence $\CL^\lc_\lam$ is ample over $\IA^1$ for $0<\lam \ll 1$. Fix a very small $\lam>0$ and let $\CL^\lc=\CL^\lc_\lam$. We get an ample test configuration $(\CX^\lc,\D_{\CX^\lc}; \CL^\lc, \frac{a}{d_1}\eta^\lc)$ such that
\begin{eqnarray*}  
\BH^g(\CX^\lc,\D_{\CX^\lc}; \CL^\lc, \frac{a}{d_1}\eta^\lc)\le 
\BH^g(\CX^{(d_1)}, \D_{\CX^{(d_1)}}; \CL^{(d_1)}, \frac{a}{d_1} \eta^{(d_1)} ). 
\end{eqnarray*}

{\bf Step 3.} (Ample configuration $\CX^{\rm ac}$). Choose $q \gg 1$ such that $\CH^\lc = \CL^\lc - (1+q)^{-1} (\CL^\lc + K_{\CX^\lc}+\D_{\CX^\lc})$ is ample over $\IA^1$. Set $\CX^0=\CX^\lc; \CL^0=\CL^\lc, \CH^0=\CH^\lc$ and $\lam_0 = 1+q$. Running a $\IG_m$-equivariant $(K_{\CX^0}+\D_{\CX^0})$-MMP with scaling $\CH^0$, we get a sequence of birational maps 
\begin{eqnarray*}
\CX^0\dashrightarrow 
\CX^1\dashrightarrow 
\cdots
\dashrightarrow \CX^k. 
\end{eqnarray*}
Let $\CH^j$ be the pushforward of $\CH^0$ to $\CX^j$, and $\lam_{j+1} = \inf\{\lam: K_{X^j}+\lam \CH^j \text{ is nef over } \IA^1\}$ be the nef threshold. Then $\CX^j\dashrightarrow \CX^{j+1}$ is the contraction of a $(K_{\CX^j}+\D_{\CX^j}+\lam_{j+1}\CH^j)$-trivial extremal ray. We have 
\begin{eqnarray*}
1+q = \lam_0 \ge \lam_1 \ge \cdots \ge \lam_k > \lam_{k+1} = 1,  
\end{eqnarray*}
where the last equality follows from the fact that the pseudo-effective threshold of $K_{\CX^0}+\D_{\CX^0}$ with respect to $\CH^0$ is $1$. 
For any $\lam>1$, we denote by
\begin{eqnarray*}  
\CL_\lam = (\lam-1)^{-1} (K_{\CX^0}+\D_{\CX^0}+\lam \CH^0), \quad
E=K_{\CX^0}+\D_{\CX^0}+\CH^0 = \sum_i e_i E_i, 
\end{eqnarray*}
with $e_1\le e_2\le \cdots \le e_l$. Then
\begin{eqnarray*}
\CL_\lam + K_{\CX^0} +\D_{\CX^0}
= \frac{\lam}{\lam-1}(K_{\CX^0}+\D_{\CX^0}+\CH^0) 
= \frac{\lam}{\lam-1} E. 
\end{eqnarray*}
Let $\CL^j_\lam$ and $E^j$ be the push-forward of $\CL_\lam$ and $E$ to $\CX^j$ respectively. And we denote by $\CF^j_\lam 
= \CF^j_{(\CX^j,\D_{\CX^j}; \CL^j_\lam, \eta^j)}$. Then for any $\lam_j\ge \lam \ge\lam_{j+1}$, we have 
\begin{eqnarray*}
\BH^g\Big(\frac{a}{d_1}\CF^j_\lam\Big)
&=& \log\Big( \int_\BO g\Big(\frac{a}{d_1} 
\big( \BL(\CF_\lam)-G_{\CF_\lam}(y) \big)
\Big) \dif y \Big) \\
&=& \log\Big( \int_\BO g\Big(\frac{a}{d_1}
\max_i\big( \frac{\lam}{\lam-1}(e_1-e_i) + A_{X,\D}(v_i) -G_{v_i}(y) \big)
\Big) \dif y \Big), \\
\frac{\dif}{\dif\lam}
\BH^g\Big(\frac{a}{d_1}\CF^j_\lam\Big)
&=&  \frac{a}{d_1}
\frac{\int_\BO (\lam-1)^{-2}(e_{i(y)}-e_1) \cdot 
g'\circ f^j(\lam,y) \dif y}{\int_\BO g\circ f^j(\lam,y) \dif y}  \ge 0.  
\end{eqnarray*}
where $f^j(\lam,y) = \frac{a}{d_1} 
\big( \BL(\CF^j_\lam)-G_{\CF^j_\lam}(y) \big)$. On the other hand, the filtration is not changed under divisorial contractions and flips. Hence for any $0\le j\le k$ we have
\begin{eqnarray*}  
\BH^g\Big(\frac{a}{d_1}\CF^j_{\lam_{j+1}}\Big) = 
\BH^g\Big(\frac{a}{d_1}\CF^{j+1}_{\lam_{j+1}}\Big). 
\end{eqnarray*}

Recall that $K_{\CX^k}+\D_{\CX^k}+\CH^k$ is nef over $\IA^1$. So is 
\begin{eqnarray*}  
K_{\CX^k}+\D_{\CX^k}+\CL^k_{\lam_k} = \frac{\lam_k}{\lam_k-1}(K_{\CX^k}+\D_{\CX^k}+\CH^k). 
\end{eqnarray*}
By negativity lemma, we have $K_{\CX^k}+\D_{\CX^k}+\CL^k_{\lam_k} \sim_{\IQ,\IA^1} 0$. Let $\CX^\ac = \CX^k $ and $\CL^\ac = \CL^{k}_{\lam_k}$. Now we get a test configuration $(\CX^\ac,\D_{\CX^\ac},\CL^\ac,\frac{a}{d_1}\eta^\ac)$ with $-(K_{\CX^\ac}+\D_{\CX^\ac}) \sim_{\IQ,\IA^1} \CL^\ac$ ample over $\IA^1$, such that 
\begin{eqnarray*}  
\BH^g(\CX^\ac,\D_{\CX^\ac},\CL^\ac,\frac{a}{d_1}\eta^\ac)\le 
\BH^g(\CX^\lc,\D_{\CX^\lc}; \CL^\lc, \frac{a}{d_1}\eta^\lc). 
\end{eqnarray*}

{\bf Step 4.} (Special test configuration $\CX^{\rm s}$). By \cite[Theorem 6]{LX14}, there exists a special test configuration $\CX^\s$ birational to $(\CX^\ac)^{(d_2)}$ over $\IA^1$ for some $d_2>0$, such that $\CX^s_0$ is a log canonical place of $((\CX^\ac)^{(d_2)},\D_{(\CX^\ac)^{(d_2)}}+(\CX^\ac)^{(d_2)}_0)$. By \cite[1.4.3]{BCHM10}, there exists a $\IG_m$-equivariant birational morphism $\pi':\CX'\to (\CX^\ac)^{(d_2)}$ which precisely extracts $\CX^\s_0$. Hence $K_{\CX'}+\D_{\CX'} = \pi'^*(K_{(\CX^\ac)^{(d_2)}} + \D_{(\CX^\ac)^{(d_2)}})$ and 
\begin{eqnarray*}
\CF_{(\CX',\D_{\CX'},-(K_{\CX'} +\D_{\CX'}), \frac{a}{d_1d_2}\eta')} = 
\CF_{(\CX^\ac,\D_{\CX^\ac},-(K_{\CX^\ac} +\D_{\CX^\ac}),\frac{a}{d_1}\eta^\ac)}. 
\end{eqnarray*}

Let $p:\hat{\CX}\to (\CX', \D_{\CX'})$ and $q:\hat{\CX}\to (\CX^\s,\D_{\CX^\s})$ be a common log resolution, and $E=-q^*(K_{\CX'}+\D_{\CX'})+p^*(K_{\CX^\s}+\D_{\CX^\s}) = \sum_i e_i E_i$ with $e_1\le \cdots \le e_l$. We denote by $\CL_\lam=-q^*(K_{\CX'}+\D_{\CX'})+\lam E$ and $\CF_\lam=\CF_{(\CX',\D_{\CX'}; \CL'_\lam, \eta')}$. Then
\begin{eqnarray*}
\BH^g\Big(\frac{a}{d_1d_2}\CF_\lam\Big)
&=& \log\Big( \int_\BO g\Big(\frac{a}{d_1d_2}
\max_i\big( \lam(e_1-e_i) + A_{X,\D}(v_i) -G_{v_i}(y) \big)
\Big) \dif y \Big), \\
\frac{\dif }{\dif \lam}
\BH^g\Big(\frac{a}{d_1d_2}\CF_\lam\Big)
&=&  \frac{a}{d_1d_2}
\frac{\int_\BO (e_1-e_{i(y)}) \cdot 
g'\circ f(\lam,y) \dif y}{\int_\BO g\circ f(\lam,y) \dif y}  \le 0.  
\end{eqnarray*}
We conclude that 
\begin{eqnarray*}
\BH^g\Big(\CX^\s,\D_{\CX^\s}, -(K_{\CX^\s} +\D_{\CX^\s}), \frac{a}{d_1d_2}\eta^\s\Big)
\le
\BH^g\Big(\CX',\D_{\CX'}, -(K_{\CX'} +\D_{\CX'}), \frac{a}{d_1d_2}\eta'\Big). 
\end{eqnarray*}
\end{proof}

\begin{rmk}\rm 
\label{Remark: G-equivariant MMP}
If $(X,\D)$ admits a connected reductive group $\IG$-action, and $(\CX,\D_\CX;\CL,a\eta)$ is a $\IG$-equivariant normal test configuration of $(X,\D)$, then the special test configuration $(\CX^s,\D_{\CX^s};\CL^s,a^s\eta^s)$ obtained above can also be $\IG$-equivariant as explained in \cite[Theorem A.1]{Li19}. 
\end{rmk}

Recall that a divisorial valuation $v$ over $(X,\D)$ is called {\it special} if $A_{X,\D}(v)< \lam_\max(v)$ and there exists a $\IQ$-complement $\Gamma$ of $(X,\D)$ such that $v$ is the unique log canonical place of $(X,\D+\Gamma)$. 
By the one-to-one correspondence of special test configurations and special divisorial valuations \cite[Theorem 4.27]{Xu24}, we have the following corollary, which is a strengthening of Theorem \ref{Theorem: gHT: Valuative criterion} in the log Fano case. 

\begin{cor}
\label{Corollary: special valuative criterion}
For any log Fano pair $(X,\D)$, we have 
\begin{eqnarray*}  
h^g(X,\D) 
\,\,\,=\,\,\, 
\inf_{v} \, \BH^g(\CF_v)
\,\,\,=\,\,\,
\inf_{v} \, \tbeta^g(v), 
\end{eqnarray*}
where $v$ runs over all the special divisorial valuations over $X$. 
\end{cor}

The second equality follows easily from Remark \ref{Remark: H^g = beta^g}.

\subsection{Existence of $\BH^g$-minimizer}

We have shown that $\BH^g$ admits at most one valuative minimizer in Corollary \ref{Corollary: gHT: uniqueness of minimizer}. In the following theorem, we prove the existence. 

\begin{thm}
\label{Theorem: gHT: Existence} 
There exists a quasi-monomial valuation $v_0$ such that 
\begin{eqnarray*}  
h^g(X,\D)
\,\,\,=\,\,\, 
\BH^g(\CF_{v_0}) 
\,\,\,=\,\,\, 
\tbeta^g(v_0). 
\end{eqnarray*}
\end{thm}

\begin{proof}
The proof is verbatim to \cite[Theorem 4.9]{HL20} with $h(X,\D)$ and $\tbeta$ replaced by $h^g(X,\D)$ and $\tbeta^g$ respectively. We briefly recall the argument. 
%By Corollary \ref{Corollary: special valuative criterion}, we can find a sequence of divisorial valuations $v_i$ such that $\lim_{i\to \infty} \tbeta^g(v_i) = h^g(X,\D)$, and for every $i$, there exists a $\IQ$-complement $\Gamma_i$ of $(X,\D)$, such that $v_i$ is a log canonical place of $(X,\D+\Gamma_i)$. 
By \cite[Theorem A.2]{BLX19} (a variant of boundedness of complements \cite{Bir19}), there exists an integer $N$ depending only on $\dim X$ and the coefficients of $\D$, such that every $\IQ$-complement of $(X,\D)$ is a $N$-complement. 

Recall $L=-(K_X+\D)$ and $R_m = H^0(X,mL)$. Let $W= \IP(R_N)$ and $D$ be the universal $\IQ$-divisor on $X\times W$ parametrizing divisors in $\frac{1}{N}|NL|$. By lower semicontinuity of log canonical thresholds (\cite[Lemma 1,42]{Xu24}), the subset $Z=\{w\in W\mid \lct(X, \D+D_w) =1\} \seq W$ is locally closed. For any $z\in Z$, we denote by 
\begin{eqnarray}
\label{Eqnarray: b_z}
b_z := \mathop{\inf}_{v\in \LC(X, \D+D_z)} \tbeta^g(v). 
\end{eqnarray}
Choose a log resolution $(Y_z,E_z) \to (X,\D+D_z)$. Then $\LC(X, \D+D_z)\seq \QM(Y,E)$. Hence the infimum in (\ref{Eqnarray: b_z}) is a minimum by Theorem \ref{Theorem: beta^g on simplicial cone}, that is, $b_z = \tbeta^g(v_z)$ for some $v_z\in \LC(X,\D+D_z)$. 

Since $(X_Z, \D_Z+D_Z) := (X\times Z, \D\times Z+ D|_{X\times Z}) \to Z$ is a $\IQ$-Gorenstein family of pairs, we can divide $Z$ into a disjoint union of finitely many locally closed subsets $Z= \sqcup_j Z_j$ such that, for each $j$, $Z_j$ is smooth, and there exists an \'etale cover $Z_j'\to Z_j$ such that the base change $(X_{Z_j'}, \D_{Z_j'}+D_{Z_j'})$ admits a fiberwise log resolution $(Y_{Z_j'},E_{Z_j'})$ over $Z_j'$. For any prime divisor $F \in \QM(Y_{Z_j'},E_{Z_j'})$, by the proof of \cite[Theorem 4.2]{BLX19} (using invariance of plurigenera \cite{HMX13}), we see that $\DH_{F_z}$ is constant for $z\in Z_j'$. Hence for any $v \in \QM(Y_{Z_j'},E_{Z_j'})$, the DH measure $\DH_{v_z}$ is constant for $z\in Z_j'$. On the other hand, $A_{X,\D}(v_z)$ is constant for $z\in Z_j'$ since $(Y_{Z_j'},E_{Z_j'})$ is simple normal crossing over $Z_j'$. We conclude that $b_z$ is constant for $z\in Z_j'$, and we denote this number by $b_j$. 

Finally, by Corollary \ref{Corollary: special valuative criterion} and by our choice of $N$ and $Z$, we have $h^g(X,\D) = \inf_{z\in Z} b_z = \min_{j} b_{j}$. Let $j_0$ be a minimizer. Then for any $z\in Z_{j_0}'$, the minimizer $v_z$ of $b_z$ in (\ref{Eqnarray: b_z}) is the desired quasi-monomial valuation minimizing $h^g(X,\D)$.  
\end{proof}

\begin{thm}
\label{Theorem: G-invariant minimizer}
If $(X,\D)$ admits a connected reductive group $\IG$-action, then the $\BH^g$-minimizer $v_0$ is $\IG$-invariant.
\end{thm}

Recall that a valuation $v$ is called $\IG$-invariant if for any $\phi\in \IG$, viewing as an automorphism $\phi: X\to X$, we have $\phi^*v = v$, where $(\phi^*v)(f) := v(f\circ\phi)$ for any rational function $f$ on $X$.  
The above theorem is a consiquence of the uniqueness of $\BH^g$-minimizer Corollary \ref{Corollary: gHT: uniqueness of minimizer}. We give another proof which does not rely on the convexity of $\BH^g$. 

\begin{proof}\rm 
This follows from the similar argument of \cite[Theorem 4.63 (i)]{Xu24}. We use the same notions as in the above proof. By Remark \ref{Remark: G-equivariant MMP} and Corollary \ref{Corollary: special valuative criterion}, we see that $h^g(X,\D)$ is approximated by a series of $\IG$-invariant special divisorial valuations $E_m$, which are log canonical places of $N$-complements. Hence $E_m$ is a log canonical place of $(X,\D+\Bs|M_m|^{\frac{1}{N}})$, where 
\begin{eqnarray*}  
M_m = \CF_{E_m}^{NA_{X,\D}(E)}R_N \seq R_N, 
\end{eqnarray*}
is a $\IG$-invarant sublinear series. Let $W$ be the subvariety of $\cup_i \gr(i,R_N)$ parametrizing $\IG$-invariant sublinear series of $R_N$, and $M\to W$ be the corresponding universal family. Also by lower semicontinuity of log canonical thresholds, we have locally closed subset $Z=\{w\in W\mid \lct(X, \D+\Bs|M_w|^{\frac{1}{N}}) =1\} \seq W$. For any $z\in Z$, we define 
\begin{eqnarray}
\label{Eqnarray: b_z, equivariant cases}
b_z := \mathop{\inf}_{v\in \LC^\IG(X, \D+\Bs|M_w|^{\frac{1}{N}})} \tbeta^g(v),  
\end{eqnarray}
where $\LC^\IG(X, \D+\Bs|M_w|^{\frac{1}{N}}) \seq \LC(X, \D+\Bs|M_w|^{\frac{1}{N}})$ consists of $\IG$-invariant valuations. Also by Theorem \ref{Theorem: beta^g on simplicial cone}, we have $b_z = \tbeta^g(v_z)$ for some $v_z\in \LC^\IG(X, \D+\Bs|M_w|^{\frac{1}{N}})$. Now the same argument of the last two paragraph of the above proof shows that $h^g(X,\D) = b_z$ for some $z\in Z$, which is minimized by the $\IG$-invariant quasi-monomial valuation $v_z$. 
\end{proof}

\subsection{Finite generation and weighted K-stability}

\begin{thm}
\label{Theorem: gHT: Finite generation} 
The minimizer $v_0$ of $\BH^g$ is special. 
\end{thm}

\begin{proof}
Choose $\CF_0=\CF_{v_0}$ in Lemma \ref{Lemma: gHT: H-minimizer v_0 and delta^(g,v_0)}. From the construction of $v_0$, we know that it is a log canonical place of some $N$-complements. So we have $\mu(\CF_{v_0}) = A_{X,\D}(v_0)$ by Lemma 2.8. In other words, $v_0$ minimizes $\lct(I_\bu^{\mu(\CF_{v_0})}(\CF_{v_0}))$. By Lemma \ref{Lemma: gHT: H-minimizer v_0 and delta^(g,v_0)}, we see that $v_0$ is a minimizer of $\delta^{g',v_0}(X,\D)=1$. Hence it is a special valuation by \cite[Theorem 5.4]{BLXZ23} (choosing $\xi=0$ and $v=v_0$). 
\end{proof}

%\begin{lem}\label{Lemma: gHT: delta^g-minimizer is quasi-monomial}Any minimizer $v$ of $\delta^{g',v_0}(X,\D)$ is quasi-monomial. \end{lem}

%\begin{proof}Following the argument of \cite[Proposition 4.8]{BJ20} (with the $T$-invariant replaced by the $S^{g',v_0}$-invariant), we see that $v$ computes $\lct(\fa_\bu(v))$. Hence $v$ is quasi-monomial by \cite[Theorem 1.1]{Xu19}. \end{proof}

By definition of special valuations Theorem \ref{Theorem: special valuation, higher rank f.g.}, we see that the $\BH^g$-minimizer $v_0$ induces a (multistep) special degeneration $(X_0,\D_0,\xi_0)$ of $(X,\D)$ with klt central fiber. We call $(X_0,\D_0,\xi_0)$ the {\it $g$-optimal degeneration} of $(X,\D)$. 
Next we study this degeneration of $(X,\D)$. We first recall some notions in the weighted K-stability theory. 

Assume that $(X,\D)$ admits a torus $\IT=\IG_m^r$-action. Then the anti-canonical ring $R_\bu=R(X,\D)=\oplus_{m\in l_0\IN}R_m$ admits a canonical weight decomposition $R_m=\oplus_{\alpha \in M} R_{m,\alpha}$, where $M= \Hom(\IT,\IG_m)\cong \IZ^r$ is the weight lattice. Let $N=M^\vee$ be the coweight lattice and $N_\IR = N\otimes_\IZ \IR$. A filtration $\CF$ is called $\IT$-{\it invariant} if $\CF^\lam R_m = \oplus_\alpha \CF^\lam R_{m,\alpha}$. 

For any $\xi\in N_\IR$ and $\IT$-invariant filtration $\CF$, the $\xi$-{\it twist} of $\CF$ is defined by 
\begin{eqnarray*}  
\CF_\xi^\lam R_{m} = \oplus_{\alpha\in M} (\CF_\xi^{\lam}R_m)_\alpha, \quad 
(\CF_\xi^{\lam}R_m)_\alpha := \CF^{\lam-\la\alpha,\xi\ra}R_{m,\alpha}. 
\end{eqnarray*}
We will simple denote the filtration $\CF_{\triv, \xi}^\lam R_m= \oplus_{\la\alpha,\xi\ra\ge \lam} R_{m,\alpha}$ by $\xi$, then  
\begin{eqnarray*}  
\mu(\xi) = \mu(\CF_{\triv, \xi}) = \mu(\CF_\triv) =0, 
\end{eqnarray*}
by the following lemma. 

\begin{lem}\cite[Lemma 6.24]{Xu24}
For any $\IT$-invariant linearly bounded filtration $\CF$ on $R$, and any $\xi\in N_\IR$, we have $\mu(\CF_\xi) = \mu(\CF)$. 
\end{lem}

Recall that $g':\IR\to \IR_{>0}$ is the first order derivative of $g$. Then for any $\xi\in N_\IR$, we may  define the $(g',\xi)$-weighted Ding invariants of $(X,\D)$. 

\begin{defi}\rm
\label{Definition: weighted Ding stability}
For any $\IT$-invariant linearly bounded filtration $\CF$ on $R$, we define the {\it $(g',\xi)$-weighted Ding invariant} by 
\begin{eqnarray*}
\BD^{g',\xi}(\CF) 
\,\,\,=\,\,\, \BD^{g',\xi}_{X,\D}(\CF) 
\,\,\,:=\,\,\,
\mu_{X,\D}(\CF) - S^{g',\xi}(\CF). 
\end{eqnarray*}
The log Fano pair $(X,\D)$ is called {\it $\IT$-equivariantly $(g',\xi)$-weighted Ding-semistable} if $\BD^{g',\xi}(\CF)\ge 0$ for any $\IT$-invariant linearly bounded filtration $\CF$ on $R$. If moreover, for any $\IT$-equivariant normal test configuration $(\CX,\D_\CX;\CL)$ of $(X,\D)$, $\BD^{g',\xi}(\CX,\D_\CX;\CL)= 0$ implies that $(\CX,\D_\CX;\CL)$ is a product TC, then $(X,\D)$ is called {\it $\IT$-equivariantly $(g',\xi)$-weighted Ding-polystable}.  

The log Fano triple $(X,\D,\xi)$ is called $g'$-{\it weighted K-(semi/poly)stable} if $(X,\D)$ is $\IT$-equivariantly $(g',\xi)$-weighted Ding-(semi/poly)stable for some $\IT$-action. By \cite[Remark 5.10]{BLXZ23}, the definition is independent of the choice of the $\IT$-action. 
\end{defi}

\begin{thm}
\label{Theorem: gHT: Weighted K-stability}
Let $v_0$ be a quasi-monomial valuation over $X$ with finitely generated associated graded ring $\gr_{v_0}R$, which induces a klt (multistep) special degeneration $(X_0,\D_{0},\xi_0)$. Then $v_0$ minimizes $\BH^g$ if and only if $(X_0,\D_{0},\xi_0)$ is $g'$-weighted K-semistable. 
\end{thm}

\begin{proof}
We follow the proof of \cite[Theorem 5.3]{HL20}. First assume that $v_0$ minimizes $\BH^g$. Denote by $(W,\D_W,\xi) = (\CX_0,\D_{\CX,0},\xi_0)$ and assume that it is $g'$-weighted K-unstable. Then by a variant of \cite{LX14} (see for example \cite[Theorem 2.43]{HL20}), there exists a special test configuration $(\CW,\D_{\CW},\eta)$ such that 
\begin{eqnarray*}  
\BD^{g',\xi}_{W,\D_W}(\CW,\D_{\CW},\eta) < 0. 
\end{eqnarray*}
We denote by $(Y,\D_Y,\eta)=(\CW_0,\D_{\CW,0},\eta)$, then
\begin{eqnarray*}  
%\Fut^{g',\xi}_{Y,\D_Y}(\eta)= 
\BD^{g',\xi}_{Y,\D_Y}(\eta)
= \BD^{g',\xi}_{W,\D_W}(\CW,\D_{\CW},\eta) 
< 0. 
\end{eqnarray*}
Then we can construct a series of valuations $\{v_\vep \}_{\vep \in \IR}$ as \cite{LX18} inducing special degenerations of $(X,\D)$ with central fibers $(Y,\D_Y, \xi+\vep\eta)$. 
Then $\BH^g_{X,\D}(v_\vep) = \BH^g_{Y,\D_Y}(\xi+\vep\eta)$ (where $\BH^g(\xi) := \BH^g(\CF_{\triv,\xi}) = \BH^g(\wt_\xi)$ by the shifting invariance of $\BH^g$ and $\CF_{\triv,\xi} = \CF_{\wt_\xi}(-A_{Y,\D_Y}(\wt_\xi))$ by \cite[Proposition 3.8]{Li19}). Since $\mu(\xi')=\mu(\CF_{\triv,\xi'})=0$ for any co-weight vector $\xi'$ on $Y$, we have 
\begin{eqnarray*}  
\BH^g_{Y,\D_Y}(\xi+\vep\eta) 
= \log\big(
\int_\BP g(-\la \alpha, \xi+\vep\eta \ra) \DH_\BP(\dif \alpha) \big). 
\end{eqnarray*}
Hence
\begin{eqnarray*}
\frac{\dif}{\dif \vep}|_{\vep = 0} \,\,
\BH^g_{X,\D}(v_\vep)
&=&
\frac{\int_\BP (-\la \alpha, \eta \ra)\cdot g'(-\la \alpha, \xi\ra) \DH_\BP(\dif \alpha)}{\int_\BP g(-\la \alpha, \xi \ra) \DH_\BP(\dif \alpha)}
\\
&=&
\frac{1}{\Bv^{g}}
\int_\BP (-\la \alpha, \eta \ra)\cdot g'(-\la \alpha, \xi\ra) \DH_\BP(\dif \alpha)
\,\,\,=\,\,\, \frac{\Bv^{g'}}{\Bv^{g}} 
\cdot \BD^{g',\xi}_{Y,\D_Y}(\eta)
\,\,\,<\,\,\, 0, 
\end{eqnarray*}
which contradicts that $v_0$ minimizes $\BH^g_{X,\D}$. 

Conversely, assume that $(W,\D_W,\xi)$ is $g'$-weighted K-semistable. Then for any linearly bounded filtration $\CF$ on $R$. We define its {\it initial term degeneration} $\CF'$ on $\gr_{v_0}R$ by
\begin{eqnarray*}
\CF'^{\lam} \gr_{v_0} R_m 
:= \la \bar{s}_i: s_i \in \CF^\lam R_m \ra, 
\end{eqnarray*}
where $\{s_i\}$ is a basis of $R_m$ which is compatible with both $v_0$ and $\CF$. Hence $\DH_\CF = \DH_{\CF'}$. By lower semicontinuity of log canonical thresholds, we have
$\mu_{X,\D}(\CF) \ge \mu_{W,\D_W}(\CF'). $
Hence 
\begin{eqnarray}
\label{Inequality: H^g ineqality}
    \BH^g_{X,\D}(\CF) 
\ge \BH^g_{W,\D_W}(\CF')
\ge \BH^g_{W,\D_W}(\xi)
=   \BH^g_{X,\D}(v_0), 
\end{eqnarray}
where the second inequality follows from the $g'$-weighted K-semistability of $(W,\D_W,\xi)$. Indeed, since $\BH^g$ is strictly convex along geodesics, it suffices to show that the derivative of $\BH^g_{X,\D}(\CF_t)$ at $t=0$ is non-negative, where $\CF_t$ is the geodesic connecting $\CF_0=\CF_{\triv, \xi}$ and $\CF_1=\CF'$. Note that
\begin{eqnarray*}
\CF^\lam_t R_m 
&=& \sum_{(1-t)\mu + t\nu\ge \lam}
\CF_0^\mu R_m \cap \CF_1^\nu R_m \\
&=& \Big\{
s\in R_m\mid (1-t)\ord_{\CF_0}(s) + t\, \ord_{\CF_1}(s) \ge \lam
\Big\} \\
&=& \bigoplus_{\alpha\in M} \Big\{
s\in R_{m,\alpha}\mid (1-t)\la\alpha,\xi\ra + t\, \ord_{\CF'}(s) \ge \lam
\Big\} \\
&=& \bigoplus_{\alpha\in M} \Big\{
s\in R_{m,\alpha}\mid t\Big(\ord_{\CF'}(s) + \la\alpha,\frac{1-t}{t}\xi\ra \Big)\ge \lam
\Big\} \\
&=& \Big\{
s\in R_m\mid \ord_{t\CF'_{\frac{1-t}{t}\xi}}(s) \ge \lam
\Big\} 
\,\,\, = \,\,\,
(t\CF'_{\frac{1-t}{t}\xi})^\lam R_m. 
\end{eqnarray*}
Hence $\CF_t = t\CF'_{\frac{1-t}{t}\xi}$. Recall that $\mu(\CF)$ is invariant under $\xi$-twist, and linear under rescaling. Hence $\mu(\CF_t) = t\mu(\CF')$. We also have $G_{\CF}(y)=(1-t)\la \alpha,\xi \ra+tG_{\CF'}(y)$ where $y=(\alpha,y')$. Hence 
\begin{eqnarray*}
\BH^g(\CF_t) 
&=& \log\Big(
\int_\BO g(\mu(\CF_t) - G_{\CF_t}(y)) \dif y
\Big) \\
&=& \log\Big(
\int_\BO g\big(-\la \alpha,\xi \ra + t( \mu(\CF') - G_{\CF'}(y)+ \la \alpha,\xi \ra)\big) \dif y
\Big) \\
&=& \log\Big(
\int_\BO g\big(-\la \alpha,\xi \ra + t( \mu(\CF'_\xi) - G_{\CF'_\xi}(y))\big) \dif y
\Big),  \\
\frac{\dif}{\dif t}|_{t=0}\, \BH^g(\CF_t) 
&=& \frac{\int_\BO g'\big(-\la \alpha,\xi \ra \big) \cdot \big(\mu(\CF'_\xi) - G_{\CF'_\xi}(y)\big) \dif y}{\int_\BO g(-\la \alpha,\xi \ra) \dif y} \\
&=& \frac{\Bv^{g'}}{\Bv^{g}} \BD^{g',\xi}_{W,\D_W}(\CF'_\xi)
\,\,\, \ge \,\,\, 0, 
\end{eqnarray*}
where $y=(\alpha, y')$. Hence the second inequality in (\ref{Inequality: H^g ineqality}) holds and the proof is finished. 
\end{proof}

\begin{rmk}\rm
If $(X,\D)$ admits a connected reductive group $\IG$-action, then by Theorem \ref{Theorem: G-invariant minimizer}, the $\BH^g$-minimizer $v_0$ is $\IG$-invariant, hence $\gr_{v_0} R$ admitting the $\IG$-action and inducing a $\IG$-equivariant (multistep) special degeneration. In other word, the $g$-optimal degeneration of $(X,\D)$ is $\IG$-equivariant. 
\end{rmk}

As a corollary, we have the following characterization of $g$-optimal degeneration. 

\begin{cor}
\label{Corollary: stable under g -optimal degeneration}
Let $(X,\D)$ be a log Fano pair admitting a torus $\IG_m^r$-action, and $\xi_0\in N_\IR$. Then the filtration $\CF_{\triv,\xi_0}$ minimizes $\BH^g$ if and only if $(X,\D,\xi_0)$ is $g'$-weighted K-semistable. 
\end{cor}

Now we can finish the proof of the main theorem in this paper. 

\begin{proof}[Proof of Theorem \ref{Theorem: Generalized Tian Conjecture}]
The existence and uniqueness of the minimizer $v_0$ of $\BH^g$ follows from Theorem \ref{Theorem: gHT: Existence} and \ref{Theorem: gHT: Convexity} respectively. The valuation is special by Theorem \ref{Theorem: gHT: Finite generation}. Moreover, the (multistep) special degeneration $(X_0,\D_{0},\xi_0)$ induced by $v_0$ is $g'$-weighted K-semistable by Theorem \ref{Theorem: gHT: Weighted K-stability}.
%the equivalence of weighted Ding-stability and weighted K-stability of log Fano pairs. 
Finally, $(X_0,\D_{0},\xi_0)$ has a unique $g'$-weighted K-polystable degeneration $(Y,\D_Y,\xi_0)$ by \cite[Theorem 1.3]{HL20}, and $(Y,\D_Y,\xi_0)$ admits a $g'$-soliton by \cite[Theorem 1.3]{BLXZ23} and \cite[Theorem 1.7]{HL23}. 
\end{proof}

\section{Examples}
\label{Section: Examples}

In this section, we give some examples that Question \ref{Question: same g -optimal degenerations?} has positive answer.

\subsection{Weighted K-stable Fano varieties for any weight function}
Let $(X,\D)$ be a log Fano pair with a $\IT=\IG_m^r$-action, $M=\Hom(\IT,\IG_m), N=M^\vee$ be the weight, coweight lattices respectively. Let $\BP\seq M_\IR$ be the moment polytope of the $\IT$-action and $\DH_\BP$ be the DH measure of the $\IT$-action on $\BP$ (see for example \cite[Section 2.5 and 3.3]{MW23}). A continuous function $g_0: \BP \to \IR_{>0}$ is called a {\it weight function} if 
\begin{eqnarray}  
\label{Eqnarray. weight function}
\int_{\BP} \alpha_i \cdot g_0(\alpha) \DH_\BP(\dif \alpha) = 0, 
\end{eqnarray}
for any $1\le i\le r$. Similar to Definition \ref{Definition: weighted Ding stability}, one can define the $g_0$-weighted K-stability and Ding-stability of the log Fano $\IT$-pair $(X,\D)$. 
In the setting of $g$-optimal degenerations, we will choose  
\begin{eqnarray*}  
g_0(\alpha) = g'(-\la \alpha, \xi_0\ra),
\end{eqnarray*}
where $\xi_0$ 
is the minimizer of $\BH^g$ on $N_\IR$. We have the following easy consequence of Corollary \ref{Corollary: stable under g -optimal degeneration}, which gives some trivial examples answering Question \ref{Question: same g -optimal degenerations?} positively. 

\begin{cor}
\label{Corollary: stable under g -optimal degeneration 2}
Assume that $(X,\D)$ is $g_0$-weighted K-polystable for any weight function $g_0$. Then $(X,\D)$ is the $g$-optimal degeneration of itself for any function $g$ satisfying (\ref{Eqnarray: function g}). 
\end{cor}

Let $(X,\D)$ be a toric log Fano pair. Then $(X,\D)$ is $g_0$-weighted K-polystable for any weight function $g_0$. Indeed, any $\IT$-invariant valuation $v=\wt_\xi$ for some $\xi\in N_\IR$. Hence 
\begin{eqnarray*}  
\BD^{g_0}(\wt_\xi) = \frac{1}{\Bv^{g_0}}\int_{\BP} (-\la\alpha,\xi\ra)\cdot g_0(\alpha) \DH_\BP(\dif \alpha)= 0.  
\end{eqnarray*}
In particular, the $g$-optimal degenerations of $(X,\D)$ are always itself. 

The following non-trivial examples follow from \cite[Example 5.5]{Wang24}. 

\begin{thm}
Any Fano threefold $X$ in the families №2.28 and №3.14 of Mori-Mukai's list is $g_0$-weighted K-polystable for any weight function $g_0$. In particular, the $g$-optimal degenerations of $X$ are always $X$ itself for any function $g$ satisfying (\ref{Eqnarray: function g}). 
\end{thm}

\subsection{Non-trivial $g$-optimal degenerations}

The Fano threefolds in the family №2.23 of Mori-Mukai's list are K-unstable and admit discrete automorphism group \cite{MT22}. Hence they could not be weighted K-semistable and admit no $g_0$-soliton \cite[(1.3)]{HL23} for any weight function $g_0$. Their optimal degenerations were determined by \cite{MW24}. It is natural to ask what are their $g$-optimal degenerations for other functions $g$ satisfying (\ref{Eqnarray: function g}). 

Recall that any Fano threefold $X$ in №2.23 is obtained by blowing up the quadric threefold $Q$ along the complete intersection $C$ of a hyperplane section $H\in |\CO_Q(1)|$ and a quadric section $Q'\in |\CO_Q(2)|$. The family №2.23 is divided into two subfamilies by the smoothness of $H$, 
\begin{itemize}
    \item $X\in $ №2.23(a), if $H\cong \IP^1\times \IP^1$, 
    \item $X\in $ №2.23(b), if $H\cong \IP(1,1,2)$.
\end{itemize}
The optimal degeneration $X_0$ of $X$ in №2.23(a) is induced by the divisorial valuation $\ord_{H}$ by \cite[Corollary 1.4]{MW24}. Hence $X_0 = \Bl_C Q_0$ where $Q_0\seq \IP^4$ is the cone over a smooth quadric surface $H\seq \IP^3$, and $C\seq H\cong \IP^1\times \IP^1$ is a biconic curve (i.e. $C\in |\CO_{\IP^1\times \IP^1}(2,2)|$).

\begin{thm}
For any Fano threefold $X$ in family №2.23(a), the $g$-optimal degenerations are always $X_0$ for any function $g$ satisfying (\ref{Eqnarray: function g}). 
\end{thm}

\begin{proof}
We need to prove that $X_0$ is the $g$-optimal degeneration of $X$ for any function $g$ satisfying (\ref{Eqnarray: function g}). This is equivalent to $\BH^g_X$ being minimized by $a\cdot \ord_H$ for some $a\in \IR_{>0}$, hence is equivalent to $(X_0,a\cdot\xi)$ being $g'$-weighted K-polystable for some $a\in \IR_{>0}$, where $\xi\in N\cong \IZ$ whose filtration is a shift of $\CF_{\ord_{H}}$. We conclude by \cite[Example 5.7]{Wang24}, which says that $X_0$ is $g_0$-weighted K-polystable for any weight function $g_0$.  
\end{proof}

\section{Conflict of interest statement}
On behalf of all authors, the corresponding author states that there is no conflict of interest.

\bibliographystyle{alpha}
\bibliography{ref}

\end{document}